\documentclass[12pt]{article}
\usepackage[hmargin={1.75truecm,2truecm},vmargin={2truecm,2truecm}]{geometry}
\usepackage[english]{babel}
\usepackage[utf8]{inputenc}
\usepackage{geometry}
\usepackage{multicol,float}
\usepackage{multirow}
\usepackage[breaklinks,colorlinks,citecolor=blue,linkcolor=blue,hyperfootnotes=false]{hyperref}
\usepackage{amsmath}
\usepackage{amsthm}
\usepackage{amsfonts}
\usepackage{amssymb}
\usepackage{graphicx}
\usepackage{xcolor}
\usepackage{arydshln}
\usepackage[affil-it]{authblk}
\usepackage{natbib}

\usepackage{xspace}
\usepackage{booktabs}
\usepackage{subcaption}
\usepackage{orcidlink}

\usepackage[textsize=small,textwidth=3.3cm]{todonotes}
\usepackage[linesnumbered,ruled,vlined]{algorithm2e}

\usepackage{tikz}
\usetikzlibrary{datavisualization.formats.functions,shadows}
\usetikzlibrary{decorations.pathreplacing}
\usetikzlibrary{shapes.geometric}
\usetikzlibrary{arrows}
\usetikzlibrary{petri,calc}

\newtheorem{theorem}{Theorem}[section]

\usepackage{cleveref}
\Crefformat{appendix}{\S#2#1#3}
\crefname{theorem}{Theorem}{Theorems}
\crefname{example}{Example}{Examples}
\crefname{observation}{Observation}{Observations}
\crefname{remark}{Remark}{Remarks}
\crefname{proposition}{Proposition}{Propositions}
\crefname{lemma}{Lemma}{Lemmas}
\crefname{corollary}{Corollary}{Corollaries}
\crefname{algocf}{Algorithm}{Algorithms}	
\crefname{table}{Table}{Tables}	
\crefname{figure}{Figure}{Figures}
\crefname{algorithm}{Algorithm}{Algorithms}
\crefname{section}{Section}{Sections}
\crefname{algorithm}{Algorithm}{Algorithms}
\makeatletter
\def\blfootnote{\xdef\@thefnmark{}\@footnotetext}
\makeatother

\title{Distributed Recursion Revisited\blfootnote{
		\textit{Email addresses:} 
		\texttt{zhangweiyang@bit.edu.cn} (Wei-Yang Zhang), 
		\texttt{dongfenglian@petrochina.com.cn} (Feng-Lian Dong),
		\texttt{weizhiwei@petrochina.com.cn} (Zhi-Wei Wei),
		\texttt{wangyanru@bit.edu.cn} (Yan-Ru Wang),
		\texttt{xuzejin@petrochina.com.cn} (Ze-Jin Xu), 
		\texttt{chenweikun@bit.edu.cn} (Wei-Kun Chen), 
		\texttt{dyh@lsec.cc.ac.cn} (Yu-Hong Dai)
	}}
\author[a]{Wei-Yang Zhang\,\orcidlink{0009-0008-8476-5276}}
\author[b,c]{Feng-Lian Dong}
\author[b,c]{Zhi-Wei Wei}
\author[a]{Yan-Ru Wang\,\orcidlink{0009-0009-6256-2328}}
\author[b,c]{Ze-Jin Xu}
\author[a]{Wei-Kun Chen\,\orcidlink{0000-0003-4147-1346}}
\author[d,e]{Yu-Hong Dai\,\orcidlink{0000-0002-6932-9512}}
\affil[a]{\small School of Mathematics and Statistics, Beijing Institute of Technology, Beijing 100081, China}
\affil[b]{\small Petrochina Planning and Engineering Institute, Beijing 100086, China}
\affil[c]{\small CNPC Laboratory of Oil \& Gas Business Chain Optimization, Beijing 100086, China}

\affil[d]{\small Academy of Mathematics and Systems Science, Chinese Academy of Sciences, Beijing 100190, China}
\affil[e]{\small School of Mathematical Sciences, University of Chinese Academy of Sciences, Beijing 100049, China}
\date{}

\newcommand{\SLP}{\text{SLP}\xspace}

\newcommand{\DR}{\text{DR}\xspace}
\newcommand{\PDR}{\text{PDR}\xspace}

\newcommand{\testSLP}{\texttt{SLP}\xspace}
\newcommand{\testDR}{\texttt{DR}\xspace}
\newcommand{\testPDR}{\texttt{PDR}\xspace}
\newcommand{\testGRB}{\texttt{GRB}\xspace}

\newcommand{\tblT}{\texttt{T}\xspace}

\newcommand{\tblObj}{\texttt{Obj}\xspace}

\newcommand{\tblObjAver}{\texttt{OR}\xspace}
\newcommand{\tblIter}{\texttt{It}\xspace}
\newcommand{\tblBV}{\texttt{BObj}\xspace}

\newcommand{\tblGap}{\texttt{G\%}\xspace}
\newcommand{\tblAve}{\texttt{Aver.}\xspace}
\newcommand{\tblid}{\texttt{id-I-L-J-K}\xspace}

\newcommand{\LP}{\text{LP}\xspace}
\newcommand{\NLP}{{NLP}\xspace}

\newcommand{\BLP}{{BLP}\xspace}

\newcommand{\solver}[1]{\textsc{#1}\xspace}
\newcommand{\gurobi}{\solver{Gurobi}}

\begin{document}
	\maketitle

\begin{abstract}
The distributed recursion (\DR) algorithm is an effective method for solving the pooling problem that arises in many applications.
It is based on the well-known P-formulation of the pooling problem, which involves the flow and quality variables; and it can be seen as a variant of the successive linear programming (\SLP) algorithm, where the linear programming (\LP) approximation problem can be transformed from the \LP approximation problem derived by using the first-order Taylor series expansion technique. 
In this paper, we first propose a new nonlinear programming (\NLP) formulation for the pooling problem involving  only  the flow variables, and show that the \DR algorithm can be seen as a direct application of the \SLP algorithm to the newly proposed formulation. 
With this new useful theoretical insight, we then develop a new variant of \DR algorithm, called penalty \DR (\PDR) algorithm, based on the proposed formulation. 
The proposed \PDR algorithm is a penalty algorithm where violations of the (linearized) nonlinear constraints are penalized in the objective function of the \LP approximation problem with the penalty terms increasing when the constraint violations tend to be large.
Compared with the \LP approximation problem in the classic \DR algorithm, the \LP approximation problem in the proposed \PDR algorithm can return a solution with a better objective value, which makes it more suitable for finding high-quality solutions for the pooling problem.
Numerical experiments on  benchmark and randomly constructed instances show that the proposed \PDR algorithm is more effective than the classic \SLP and \DR algorithms in terms of finding a better solution for the pooling problem.
\end{abstract}


\section{Introduction}

The pooling problem, introduced by \citet{Haverly1978}, is a class of network flow problems on a directed graph with three layers of nodes (i.e., input nodes, pool nodes, and output nodes).
The problem involves routing flow from input nodes, potentially through intermediate pool nodes, to output nodes.
The flow originating from the input nodes has known qualities for certain attributes.
At the pool or output nodes, the incoming flows are blended, with the attribute qualities mixing linearly;
that is, the attribute qualities at a node are mixed in the same proportion as the incoming flows.
The goal of the problem is to route the flow to maximize the net profit while requiring the capacity constraints at the nodes and arcs to be satisfied and meeting the requirements of attribute qualities at the output nodes.
The pooling problem arises in a wide variety of applications, including petrochemical refining \citep{Baker1985,DeWitt1989,Amos1997}, 
wastewater treatment \citep{Galan1998,Misener2010,Kammammettu2020}, 
natural gas transportation \citep{Rios-Mercado2015,Romo2009}, 
open-pit mining \citep{Blom2014,Boland2017}, 
and animal feed problems \citep{Grothey2023}.

Due to its wide applications, various algorithms have been proposed to solve the pooling problem in the literature.
	For global algorithms that guarantee to find an optimal solution for the problem, we refer to the branch-and-bound algorithms \citep{Foulds1992,Tawarmalani2002,Misener2011} and Lagrangian-based algorithms \citep{Floudas1990,BenTal1994,Adhya1999,Almutairi2009}.
	For local algorithms that aim to find a high-quality feasible solution for the problem, 
	we refer to the discretization methods \citep{Pham2009,Dey2015,Castro2023},
	the successive linear programming (\SLP)-type algorithms \citep{Haverly1978,Lasdon1979}, 
	the variable neighborhood search \citep{Audet2004}, the construction heuristic \citep{Alfaki2014}, and the generalized Benders decomposition heuristic search \citep{Floudas1990a}.
	
The algorithm of interest in this paper is the \DR algorithm, which was first proposed in the 1970s and has been widely investigated in the literature \citep{Haverly1978,Lasdon1979,White1980,Lasdon1990,Fieldhouse1993,Kutz2014,Khor2017}.
The \DR algorithm is an \SLP-type algorithm 
which begins with an initial guess of the attribute qualities, solves an \LP approximation subproblem of the P-formulation \citep{Haverly1978} (a bilinear programming (\BLP) formulation involving flow and quality variables), and takes the optimal solution of the \LP approximation subproblem as the next iterate for the computation of the new attribute qualities. 
The process continues until a fixed point is reached.
The success of the \DR algorithm lies in its ``accurate'' \LP approximation subproblems that provide a critical connection
between quality changes at the inputs nodes and those at the output nodes \citep{Kutz2014,Khor2017}. 
In particular, the \DR algorithm first keeps track of the difference in attribute quality values multiplied by the total amount of flow in each pool (called \emph{quality error}) and distributes the error to the linear approximation of the bilinear terms, corresponding to the output nodes, in proportion to the amount of flow.
Due to the accuracy of the \LP approximation subproblems, the \DR algorithm usually finds a high-quality (or even global) solution in practice \citep{Fieldhouse1993,Kutz2014,Khor2017}. 
The \DR algorithm has become a standard in \LP modeling systems for refinery planning; many refinery companies (e.g., Chevron \citep{Kutz2014}, Aspen PIMS \citep{PIMS}, and Haverly System \citep{Kathy2022}) have applied it to solve real-world pooling problems.

It is well-known that the \DR algorithm is closely related to the classic \SLP algorithm \cite{Griffith1961}, where the \LP subproblem is derived by using the first-order Taylor series expansion technique \citep{Lasdon1990,Kathy2022}. 
In analogy to the \DR algorithm, the classic \SLP algorithm begins with an initial solution, solves the \LP approximation subproblem derived around the current iterate, and takes the optimal solution from the \LP approximation subproblem as the next iterate. 
\citet{Lasdon1979,Baker1985,Zhang1985,Greenberg1995} applied the classic \SLP algorithm to solve the P-formulation of the pooling problem, where the bilinear terms of flow and attribute quality variables are linearized by using the first-order Taylor series expansion technique.
In \citet{Lasdon1990}, the authors showed that if the amounts of flow in all pools are positive at a given solution, then
the \LP approximation subproblem in the \DR algorithm can be transformed from  that in the classic \SLP algorithm (through a change of variables); see also \citet{Kathy2022}.
In  \cref{subsec:DR}, by taking into account the case that the amount of flow in some pool may be zero, we present a (variant of) \DR algorithm whose \LP approximation subproblem is more accurate than the aforementioned \LP approximation subproblems. 
It is worthy mentioning that recognizing the theoretical relation of the \DR algorithm to the \SLP algorithm can further provide  insight into how the \DR algorithm works, thereby improving users' confidence in accepting the \DR algorithm \citep{Lasdon1990,Kathy2022}. 

\subsection{Contributions}

The goal of this paper is to provide a more in-depth theoretical analysis of the \DR algorithm and to develop more effective variants of the \DR algorithm for finding high-quality solutions for the pooling problem.
More specifically, 
\begin{itemize}
	\item 
	By projecting the quality variables out from the P-formulation, we first propose a new \NLP formulation for the pooling problem.
	Compared with the classic P-formulation, the new \NLP formulation involves only the flow variables and thus is more amenable to algorithmic design.
	Then, we develop an \SLP algorithm based on the newly proposed formulation and show that it is equivalent to the well-known \DR algorithm. 
	This enables to provide a new theoretical view on the well-known \DR algorithm, that is, it can be seen as a direct application of the \SLP algorithm to the proposed \NLP formulation. 
	\item We then go one step further to develop a new variant of \DR algorithm, called penalty \DR (\PDR) algorithm, based on the newly proposed formulation.
	The proposed \PDR algorithm is a penalty algorithm where the violations of the (linearized) nonlinear constraints are penalized in the objective function of the \LP approximation problem with the penalty terms increasing when the constraint violations tend to be large.
	Compared with the \LP approximation problem in the classic \DR algorithm, the \LP approximation problem in the proposed  \PDR algorithm can return a solution with a better objective value, which makes the proposed \PDR algorithm more likely  to find high-quality solutions for the pooling problem. 
\end{itemize}
Computational results demonstrate that the newly proposed \PDR is more effective than the classic \SLP and \DR algorithms in terms of finding high-quality feasible solutions for the pooling problem.

The rest of the paper is organized as follows.
\cref{sect:problem-description} presents the  P-formulation of the pooling problem and revisits the \DR algorithm.
\cref{section:3} develops a new \NLP formulation for the pooling problem and new \DR algorithms based on this new formulation.
\cref{sect:compu} reports the computational results.
Finally, \cref{sect:conclusion} draws the conclusion.

\section{Problem formulation and distributed recursion}
\label{sect:problem-description}

In this section, we review the P-formulation and the \DR  algorithm for the pooling problem \cite{Haverly1978}.
We also discuss the connection between the \DR and \SLP algorithms \cite{Griffith1961}.
	
\subsection{Mathematical formulation}

Let $G = (N, A)$ be a simple acyclic directed graph, where $N$ and $A$ represent the sets of nodes and directed arcs, respectively.
The set of nodes $N$ can be partitioned into three disjoint subsets $I$, $L$, and $J$, where $I$ is the set of input nodes, $L$ is the set of pool nodes, and $J$ is the set of output nodes. 
The set of directed arcs $A$ is assumed to be a subset of $(I \times L)\cup (I \times J)\cup (L \times J)$.
In this paper, we concentrate on the standard pooling problem, which does not include arcs between the pool nodes; 
for investigations on generalized pooling problem which allows arcs between pools, we refer to \citet{Audet2004}, \citet{Meyer2006}, \citet{Misener2011}, and \citet{Dai2018} among many of them.

For each $t \in N$, let $u_t$ be the capacity of this node.
Specifically, $u_i$ represents the total available supply of raw materials at input node $i \in I$, and 
$u_{\ell}$ represents the processing capability of pool $\ell \in L$, and $u_j$ represents the maximum product demand at the output node $j \in J$.
The maximum flow that can be carried on arc $(i,j)\in A$ is denoted as  $u_{ij}$.
For each arc $(i,j) \in A$, we denote by $w_{ij}$ the  weight of sending a unit flow from node $i$ to node $j$. 
Usually, we have $w_{it} \leq 0 $ for $(i,t) \in A$ with $i \in I$ and $w_{tj} \geq 0$ for $(t,j) \in A$ with $j \in J$, which reflect the unit costs of purchasing raw materials at the input nodes and the unit revenues from selling products at output nodes \citep{Dey2015}.

Let $K$ be the set of attributes.
The attribute qualities of input nodes are assumed to be known and are denoted by $\lambda_{ik}$ for $i \in I$ and $k \in K$.
For each output node $j \in J$, the lower and upper bound requirements on attribute $k$ are denoted by $\lambda^{\min}_{jk}$ and $\lambda^{\max}_{jk}$, respectively.

Let $y_{ij}$ be the amount of flow on arc $(i,j) \in A$ and  $\alpha_{jk}$ be the quality of attribute $k$ at a pool or an output node $j \in L \cup J$.
Throughout, we follow  \citet{Gupte2017} to write equations using the flow variables $y_{ij}$ with the understanding that $y_{ij}$ is defined only for $(i,j) \in A$.
The pooling problem attempts to route the flow to maximize the total weight (or net profit) while requiring the capacity constraints at the nodes and arcs to be satisfied and the attribute qualities at the output nodes $\alpha_{jk}$ to be within the range $[\lambda^{\min}_{jk}, \lambda^{\max}_{jk}]$ ($j \in J$ and $k \in K$). 
Its mathematical formulation can be written as follows:

\begin{equation}\label{prob:pooling-P}\tag{P}
	\max_{y,\, \alpha} \left\{ \sum_{(i,j) \in A} w_{ij} y_{ij} \, : \, \eqref{cons:flow-conservation}\text{--}\eqref{cons:capacity-arc}\right\},
\end{equation}
where

	\begin{align}
		& \sum_{i \in I} y_{i\ell} = \sum_{j \in J} y_{\ell j}, \ \forall \ \ell \in L, \label{cons:flow-conservation} \\
		& \sum_{i \in I} \lambda_{ik} y_{i\ell}  = \alpha_{\ell k} \sum_{j \in J} y_{\ell j}, \ \forall \ \ell \in L, \ k \in K, \label{cons:material-balance-ell} \\
		& \sum_{i \in I} \lambda_{ik} y_{ij} + \sum_{\ell \in L} \alpha_{\ell k} y_{\ell j}  = \alpha_{jk} \sum_{i \in I \cup L} y_{ij}, \ \forall \ j \in J, \ k \in K, \label{cons:material-balance-j} \\
		& \lambda_{jk}^{\min} \le \alpha_{jk} \le \lambda_{jk}^{\max}, \ \forall \ j \in J, \ k \in K, \label{cons:alpha-lb-ub-j} \\
		& \sum_{j \in L \cup J} y_{ij} \le u_i, \ \forall \ i \in I, \label{cons:capacity-node-i} \\
		& \sum_{j \in J} y_{\ell j} \le u_\ell, \ \forall \ \ell \in L, \label{cons:capacity-node-ell} \\
		& \sum_{i \in I \cup L} y_{ij} \le u_j, \ \forall \ j \in J, \label{cons:capacity-node-j} \\
		& 0 \le y_{ij} \le u_{ij}, \ \forall\ (i, j) \in A.\label{cons:capacity-arc}
	\end{align}

The flow conservation constraints \eqref{cons:flow-conservation} ensure that the total inflow is equal to the total outflow at each pool node.
Constraints \eqref{cons:material-balance-ell}--\eqref{cons:material-balance-j} ensure that for each pool or output node and for each attribute, the attribute quality of the outflows or products is a weighted average of the attribute qualities of the inflows.
Constraints \eqref{cons:alpha-lb-ub-j} require the attribute qualities of the end products at the output nodes to be within the prespecified bounds.
Constraints \eqref{cons:capacity-node-i}, \eqref{cons:capacity-node-ell}, and \eqref{cons:capacity-node-j} model the available raw material supplies, pool capacities, and maximum product demands, respectively. 
Finally, constraints \eqref{cons:capacity-arc} enforce the bounds for the flow variables.

Observe that constraints \eqref{cons:material-balance-j} and \eqref{cons:alpha-lb-ub-j} can be combined and substituted by the following linear constraints

	\begin{align}
		& \sum_{i \in I} \lambda_{ik} y_{ij}
		+ \sum_{\ell \in L} \alpha_{\ell k} y_{\ell j} \ge 
		\lambda_{jk}^{\min} \sum_{i \in I \cup L} y_{ij}, \ \forall \ j \in J, \ k \in K,
		\label{cons:material-balance-j-lb}\\
		& \sum_{i \in I} \lambda_{ik} y_{ij}
		+ \sum_{\ell \in L} \alpha_{\ell k} y_{\ell j} \le
		\lambda_{jk}^{\max} \sum_{i \in I \cup L} y_{ij}, \ \forall \ j \in J, \ k \in K.
		\label{cons:material-balance-j-ub}
	\end{align}

In addition, the quality variables $\alpha_{jk}$ for $j \in J$ and $k \in K$ can be eliminated from the problem formulation.
In the following, unless otherwise specified, we will always apply the above transformation to  formulation \eqref{prob:pooling-P}.

The nonconvex \BLP formulation \eqref{prob:pooling-P} was first proposed by \citet{Haverly1978} and often referred as \emph{P-formulation} in the literature
\citep{Tawarmalani2002,Alfaki2013,Gupte2017}.

In addition to the P-formulation \eqref{prob:pooling-P} for the pooling problem, other formulations have also been proposed in the literature, which include  the Q-formulation \citep{BenTal1994},  {PQ}-formulation \citep{Sahinidis2005},  hybrid formulation \citep{Audet2004}, {STP}-formulation \citep{Alfaki2013}, and QQ-formulation \citep{Grothey2023}. 
It is worthwhile remarking that the P-formulation \eqref{prob:pooling-P} has been widely used in refinery companies such as Chevron \citep{Kutz2014}, Aspen PIMS \citep{PIMS}, and Haverly System \citep{Kathy2022}.
This wide applicability could be attributed to the success of the \DR algorithm \citep{Haverly1978,Kathy2022}, as will be detailed in the next subsection.

\subsection{Distributed recursion algorithm}
\label{subsec:DR}

We begin with the \SLP algorithm,  which was proposed by \citet{Griffith1961} and has been frequently used to tackle the pooling problem in commercial applications \citep{Haugland2010}.
Consider the following \NLP problem:
\begin{equation}\label{prob:nlp}\tag{NLP}
	\max_{x} \left\{ f(x): g(x) \le 0\right\},
\end{equation}
where $f(x): \mathbb{R}^n \rightarrow \mathbb{R}$ and $g(x): \mathbb{R}^n \rightarrow \mathbb{R}^m$ are continuously differentiable.
The idea behind the \SLP algorithm is to first approximate \eqref{prob:nlp} at the current iterate $x^t$ by an \LP problem and then use the maximizer of the approximation problem to define a new iterate $x^{t+1}$. 
Specifically, given $x^t \in \mathbb{R}^n$, \SLP solves the following \LP problem (derived by using the first-order Taylor series expansion technique):
\begin{equation}\label{prob:nlp-Delta}\tag{$\LP(x^t)$}
	\max_{x} \left\{ f(x^{t}) + \nabla f(x^{t})^\top (x - x^t)  : g(x^{t}) + \nabla g(x^{t})^\top (x - x^t) \le 0 \right\},
\end{equation}
obtaining an optimal solution $x^{t+1}$ treated as a new iterate for the next iteration.
This procedure continues until the convergence to a fixed point is achieved, i.e., $x^{t+1} = x^{t}$. 

In order to apply the \SLP algorithm to the pooling problem \eqref{prob:pooling-P}, we consider the first-order Taylor series expansion of $\alpha_{\ell k} y_{\ell j}$ at point $(\alpha^t, y^t)$:
\begin{equation}\label{tmp-approx}
	\alpha_{\ell k} y_{\ell j} \approx \alpha_{\ell k}^t y_{\ell j}^t + \alpha_{\ell k}^t (y_{\ell j} - y_{\ell j}^t)  + y_{\ell j}^t (\alpha_{\ell k} - \alpha_{\ell k}^t)
	= \alpha_{\ell k}^t y_{\ell j} + y_{\ell j}^t (\alpha_{\ell k} - \alpha_{\ell k}^t), 
\end{equation}
and derive the \LP approximation of \eqref{prob:pooling-P} around point $(\alpha^t, y^t)$: 
\begin{equation}
	\tag{$\text{SLP}(\alpha^t,y^t)$}\label{prob:pooling-P-SLP} 
	\max_{y, \, \alpha} \left\{ \sum_{(i,j) \in A} w_{ij} y_{ij} : 
	\eqref{cons:flow-conservation}, ~
	\eqref{cons:capacity-node-i}\text{--}
	\eqref{cons:capacity-arc}, ~
	\eqref{SLP:material-balance-ell}\text{--}
	\eqref{SLP:material-balance-j-ub}
	\right\},
\end{equation} 
where 
\begin{align}
	&\sum_{i \in I} \lambda_{ik} y_{i \ell} = \alpha_{\ell k}^t \sum_{j \in J} y_{\ell j} + (\alpha_{\ell k} - \alpha_{\ell k}^t) \sum_{j \in J} y_{\ell j}^t, \ \forall \ \ell \in L, \ k \in K,
	\label{SLP:material-balance-ell}\\
	& \sum_{i \in I} \lambda_{ik} y_{ij}
	+ \sum_{\ell \in L}(\alpha_{\ell k}^t y_{\ell j} + y_{\ell j}^t (\alpha_{\ell k} - \alpha_{\ell k}^t)) \ge 
	\lambda_{jk}^{\min} \sum_{i \in I \cup L} y_{ij}, \ \forall \ j \in J, \ k \in K,
	\label{SLP:material-balance-j-lb}\\
	& \sum_{i \in I} \lambda_{ik} y_{ij}
	+ \sum_{\ell \in L} (\alpha_{\ell k}^t y_{\ell j} + y_{\ell j}^t (\alpha_{\ell k} - \alpha_{\ell k}^t)) \le
	\lambda_{jk}^{\max} \sum_{i \in I \cup L} y_{ij}, \ \forall \ j \in J, \ k \in K.
	\label{SLP:material-balance-j-ub}
\end{align}
The algorithmic details of the \SLP algorithm based on formulation \eqref{prob:pooling-P} are summarized in \cref{alg:SLP0}.

\begin{algorithm}
	\caption{The successive linear programming algorithm based on formulation \eqref{prob:pooling-P}}
	\KwIn{Choose an initial solution $(\alpha^0, y^0)$ and a maximum number of iterations $t_{\max}$.}
	\label{alg:SLP0}
	Set $t \leftarrow 0$\;
	\While{$t \leq t_{\max}$}{
		Solve \eqref{prob:pooling-P-SLP} to obtain a new iterate $(\alpha^{t+1}, y^{t+1})$\;
		\If{$(\alpha^{t+1}, y^{t+1}) = (\alpha^t, y^t)$}{
			\textbf{Stop} with a feasible solution $(\alpha^{t+1}, y^{t+1})$ of formulation \eqref{prob:pooling-P}\;
		}
		Set $t \leftarrow t+1$\;
	}
\end{algorithm}

Unfortunately, the above direct application of  the \SLP algorithm to the pooling problem usually fails to find a high-quality solution as it fails to address the {strong relation} between variables $\alpha$ and $y$ in the original formulation \eqref{prob:pooling-P}.
Indeed, by constraints \eqref{cons:material-balance-ell}, the relations  $\alpha_{\ell k} =\frac{\sum_{i \in I} \lambda_{ik} y_{i\ell}}{\sum_{j \in J} y_{\ell j}}$ for $\ell \in L$ and $k \in K$ between variables $\alpha$ and $y$ must hold (when $\sum_{j \in J} y_{\ell j} > 0$).
However, even if such relations hold at the current point $(\alpha^{t}, y^{t})$, it may be violated by the next {iterate} $(\alpha^{t+1}, y^{t+1})$ (i.e., the optimal solution of \eqref{prob:pooling-P-SLP}).
The \DR algorithm can better address the above weakness of the \SLP algorithm \citep{Haverly1978,Kathy2022}. 
Its basic idea is to project variables $\alpha$ out from the \LP approximation problem \eqref{prob:pooling-P-SLP} and use a more \emph{accurate} solution $\alpha^{t+1}$ in a postprocessing step (so that the relations $\alpha_{\ell k} =\frac{\sum_{i \in I} \lambda_{ik} y_{i\ell}}{\sum_{j \in J} y_{\ell j}}$ for $\ell\in L$ and $k \in K$ hold at the new iterate $(\alpha^{t+1}, y^{t+1})$ when $\sum_{j \in J} y_{\ell j}^{t+1} >0$). 

To project variables $\alpha$ out from the \LP approximation problem \eqref{prob:pooling-P-SLP},  we can use \eqref{SLP:material-balance-ell} and rewrite 
the linearization of $\alpha_{\ell k} y_{\ell j}$ in \eqref{tmp-approx} as
\begin{equation*}
	\begin{aligned}
		\alpha_{\ell k} y_{\ell j} 
		& \approx  \alpha_{\ell k}^t y_{\ell j} + y_{\ell j}^t (\alpha_{\ell k} - \alpha_{\ell k}^t)\\
		& \overset{\text{(a)}}{=}\alpha_{\ell j}^t y_{\ell j} + \frac{y_{\ell j}^t}{\sum_{r \in J} y_{\ell r}^t}   \left((\alpha_{\ell k} - \alpha_{\ell k}^t) \sum_{r \in J} y_{\ell r}^t \right)\\
		& = \alpha_{\ell k}^t y_{\ell j} + \frac{y_{\ell j}^t}{\sum_{r \in J} y_{\ell r}^t}   \left(\sum_{i \in I} \lambda_{ik} y_{i \ell} - \alpha_{\ell k}^t \sum_{r \in J} y_{\ell r}\right),
	\end{aligned}
\end{equation*}
Observe that (a) holds only when $\sum_{r \in J} y_{\ell r}^t > 0$. 
For the case $\sum_{r \in J} y_{\ell r}^t =0$, we have $y^t_{\ell j} = 0$ (as $j \in J$ and $y_{\ell r}^t \geq 0$ for all $r \in J$), and by \eqref{tmp-approx}, we can use the  approximation $\alpha_{\ell k} y_{\ell j} \approx \alpha_{\ell k}^t y_{\ell j}$ instead.
Combining the two cases, we obtain
\begin{equation}\label{tmp-approx2}
	\alpha_{\ell k} y_{\ell j} \approx \sigma_{\ell j k}(y): = 
	\left\{
	\begin{aligned}
		& \alpha_{\ell k}^t y_{\ell j} + \frac{y_{\ell j}^t}{\sum_{r \in J} y_{\ell r}^t}   \left(\sum_{i \in I} \lambda_{ik} y_{i \ell} - \alpha_{\ell k}^t \sum_{r \in J} y_{\ell r}\right), 
		&& \text{if} ~ \sum_{r \in J} y_{\ell r}^t > 0; \\
		& \alpha_{\ell k}^t y_{\ell j}, 
		&& \text{otherwise}.
	\end{aligned}
	\right.
\end{equation}
Observe that the linear term $\sigma_{\ell j k}(y)$ in \eqref{tmp-approx2} depends on the current iterate $(\alpha^t, y^t)$ but we omit this dependence for notations convenience.
By substituting $\alpha_{\ell k} y_{\ell j}$ with the linear terms  $\sigma_{\ell j k}(y)$ into constraints \eqref{cons:material-balance-j-lb} and \eqref{cons:material-balance-j-ub}, we obtain 
	\begin{align}
		& \sum_{i \in I} \lambda_{ik} y_{ij}
		+ \sum_{\ell \in L} \sigma_{\ell j k}(y) \ge 
		\lambda_{jk}^{\min} \sum_{i \in I \cup L} y_{ij}, \ \forall \ j \in J, \ k \in K,
		\label{DR:material-balance-j-lb}\\
		& \sum_{i \in I} \lambda_{ik} y_{ij}
		+ \sum_{\ell \in L} \sigma_{\ell j k}(y) \le
		\lambda_{jk}^{\max} \sum_{i \in I \cup L} y_{ij}, \ \forall \ j \in J, \ k \in K,
		\label{DR:material-balance-j-ub}
	\end{align}
and a new \LP approximation of problem \eqref{prob:pooling-P} at point $(\alpha^t, y^t)$: 
\begin{equation}
	\tag{$\text{DR}(\alpha^t,y^t)$}\label{prob:pooling-P-DR} 
	\max_{y} \left\{ \sum_{(i,j) \in A} c_{ij} y_{ij} : 
	\eqref{cons:flow-conservation}, ~
\eqref{cons:capacity-node-i}\text{--}
\eqref{cons:capacity-arc}, ~
\eqref{DR:material-balance-j-lb},~ \eqref{DR:material-balance-j-ub}
	\right\}.
\end{equation}

Note that in the new \LP approximation problem \eqref{prob:pooling-P-DR}, we do not need the $\alpha $ variables. 
\citet{Lasdon1990} showed the equivalence between problems \eqref{prob:pooling-P-DR} and \eqref{prob:pooling-P-SLP} when  $\sum_{j \in J} y_{\ell j}^{t} >0$ holds for all $\ell\in L$.
However, when $\sum_{j \in J} y_{\ell j}^{t} =0$ holds for some $\ell\in L$, the two problems may not be equivalent.
Indeed, for $\ell \in L$ with $\sum_{j \in J} y_{\ell j}^t = 0$, constraint \eqref{SLP:material-balance-ell} in problem \eqref{prob:pooling-P-SLP} reduces to $\sum_{i \in I} \lambda_{ik} y_{i \ell} = \alpha_{\ell k}^t \sum_{j \in J} y_{\ell j} $; and 
different from \eqref{prob:pooling-P-SLP}, problem \eqref{prob:pooling-P-DR} does not include such constraints.
It is worthwhile remarking that  these constraints enforce either the amount of flow at pool is $0$ or the qualities of attributes $k \in K$ at pool $\ell$ are fixed to $ \alpha_{\ell k}^t$, which, however, are unnecessary requirements on the flow variables and thus may lead to an inaccurate \LP approximation problem. 
As an example, if for some $k \in K$, $\alpha_{\ell k}^t < \lambda_{ik}$ holds for all $i \in I$, then the constraint $\sum_{i \in I} \lambda_{ik} y_{i \ell} = \alpha_{\ell k}^t \sum_{j \in J} y_{\ell j} $ in problem \eqref{prob:pooling-P-SLP} enforces $y_{i\ell} =0$ for all $i \in I$ and $y_{\ell j}=0$ for all $j \in J$ and thus blocks the opportunity to use pool $\ell$ \citep{Greenberg1995}.
Due to this, we decide to not include these unnecessary constraints into the \LP approximation problem \eqref{prob:pooling-P-DR}.

Solving problem \eqref{prob:pooling-P-DR} yields a solution $y^{t+1}$, which can be used to compute the quality values $\alpha_{\ell k}^{t+1}= q_{\ell k}(y^{t+1})$ for the next iteration, where $\{q_{\ell k}(y)\}$ are defined by 
\begin{equation}\label{def:quality}
	q_{\ell k}(y):= \left\{
	\begin{aligned}
		& \frac{\sum_{i \in I} \lambda_{ik} y_{i\ell}}{\sum_{j \in J} y_{\ell j}},
		&& \text{if} ~ \sum_{j \in J} y_{\ell j} > 0;\\
		&0 , && \text{otherwise},
	\end{aligned}
	\right. \quad \ \forall \ \ell \in L, \ k \in K.
\end{equation}
Note that this enforces the strong relations  $\alpha_{\ell k} =\frac{\sum_{i \in I} \lambda_{ik} y_{i\ell}}{\sum_{j \in J} y_{\ell j}}$ for $\ell \in L$ and $k \in K$ between variables $\alpha$ and $y$ (when $\sum_{j \in J} y_{\ell j} > 0$).
Also note that to ensure such relations, the second case in \eqref{def:quality} can take an arbitrary value but we decide to set it to zero for simplicity of discussion.
Also note that if $\alpha^t = q(y^t)$, then $\alpha^t_{\ell k}=0$ holds for all $\ell \in L$ with $ \sum_{j \in J} y^t_{\ell j} = 0$ and $k \in K$, and thus \eqref{tmp-approx2} reduces to
\begin{equation}\label{sigmaapprox}
	\sigma_{\ell j k}(y) = 
	\left\{
	\begin{aligned}
		& \alpha_{\ell k}^t y_{\ell j} + \frac{y_{\ell j}^t}{\sum_{r \in J} y_{\ell r}^t}   \left(\sum_{i \in I} \lambda_{ik} y_{i \ell} - \alpha_{\ell k}^t \sum_{r \in J} y_{\ell r}\right), 
		&& \text{if} ~ \sum_{r \in J} y_{\ell r}^t > 0; \\
		& 0, 
		&& \text{otherwise}.
	\end{aligned}
	\right.
\end{equation} 

The overall algorithmic details of the \DR algorithm are summarized in  \cref{alg:DR-pooling-P}.

\begin{algorithm}[htbp]
	\caption{The distributed recursion algorithm}
	\label{alg:DR-pooling-P}
	\KwIn{Choose an initial solution $y^0$ and a maximum number of iterations $t_{\max}$.}
	Set $t \leftarrow 0$\;
	\While{$t \leq  t_{\max}$}{
		For each $\ell \in L$ and $k \in K$, compute $\alpha^{t}_{\ell k} = q_{\ell k} (y^t)$, where $q_{\ell k}(y)$ is defined in \eqref{def:quality}\;
		Solve \eqref{prob:pooling-P-DR} to obtain a new iterate $y^{t+1}$\;
		\If{$y^{t+1} = y^t$}{
			\textbf{Stop} with a feasible solution $(\alpha^{t+1}, y^{t+1})$ of formulation \eqref{prob:pooling-P}\;
		}
		Set $t \leftarrow t+1$\;
	}
\end{algorithm}

Two remarks on the \DR algorithm are in order.

First, an initial solution of $y^0$ in \cref{alg:DR-pooling-P} can be constructed via solving the linear multicommodity network flow problem
\begin{equation}\label{mcf}
	\max_{y} \left\{ \sum_{(i,\,j) \in A} w_{ij} y_{ij} : 
	\eqref{cons:flow-conservation}, ~
	\eqref{cons:capacity-node-i}\text{--}
	\eqref{cons:capacity-arc}\right\},
\end{equation}
an \LP relaxation of problem \eqref{prob:pooling-P} obtained by dropping quality variables $\alpha$ and all nonlinear quality constraints \eqref{cons:material-balance-ell}, \eqref{cons:material-balance-j-lb}, and \eqref{cons:material-balance-j-ub} from problem \eqref{prob:pooling-P} \citep{Greenberg1995}.

Other sophisticated techniques for the construction of the initial solution can be found in \citet{Audet2004,Haverly1978,Dai2018}.

Second, in order to provide more insights of the \DR algorithm from a practical perspective, let us rewrite {the first case} of the approximation \eqref{tmp-approx2} at point $(\alpha^t, y^t)$ into $\alpha_{\ell k} y_{\ell j} \approx \alpha_{\ell k}^t y_{\ell j} + \beta_{\ell j}^t R_{\ell k}$, where 
\begin{align}\label{def:R}
	& R_{\ell k} = \sum_{i \in I} \lambda_{ik} y_{i \ell} - \alpha_{\ell k}^t \sum_{j \in J} y_{\ell j}, 
	\ \forall  \ k \in K,\\
	& \label{def:beta}
	\beta_{\ell j}^t = \frac{y_{\ell j}^t}{\sum_{r \in J} y_{\ell r}^t},~~ \ \forall  \ j \in J.
\end{align}
The error term $R_{\ell k}$ in \eqref{def:R} characterizes the difference in quality value times the total amount of flow in pool $\ell$ while 
the distribution factor $\beta^t_{\ell j}$ represents the proportion to the amount of flow in pool $\ell$ that terminates at output node $j$.
Observe that $\sum_{j \in J} \beta_{\ell j}^t=1$.

Thus, the approximations $\alpha_{\ell k} y_{\ell j} \approx \alpha_{\ell k}^t y_{\ell j} + \beta_{\ell j}^t R_{\ell k}$ enforce that the error term $R_{\ell k}$ is distributed to each output node in proportion to the amount of flow. 

\section{Reformulation and new successive linear programming algorithms}\label{section:3}

The \DR algorithm is indeed an \SLP-type algorithm where in each iteration, it first solves the ``projected'' \LP problem \eqref{prob:pooling-P-DR}, obtained by projecting variables $\alpha_{\ell k}$ out from the \LP approximation subproblem of the original problem \eqref{prob:pooling-P} (when $\sum_{r \in J} y_{\ell r}^t > 0$) and removing constraints \eqref{SLP:material-balance-ell} (when $\sum_{r \in J} y_{\ell r}^t =0$), to compute the flow values $y^{t+1}$, and then uses \eqref{def:quality} to obtain the quality values $\alpha^{t+1}$.
In this section, we go for a different direction by directly projecting the quality variables $\alpha$ out from the \NLP formulation \eqref{prob:pooling-P}, obtaining a new \NLP formulation that includes only the $y$ variables. 
Subsequently, we propose two \SLP-type algorithms based on this new proposed  \NLP formulation.

\subsection{Flow formulation}\label{subsect:flow-formulation}

Formulation \eqref{prob:pooling-P} involves the $y$ and $\alpha$ variables, which represent the flows on the arcs and the qualities of the pool components, respectively.
However, as discussed in \cref{subsec:DR}, there exist strong relations between the $\alpha$ and $y$ variables in formulation \eqref{prob:pooling-P}.

Indeed, letting $(\alpha, y)$ be a feasible solution of formulation \eqref{prob:pooling-P}, for a pool $\ell \in L$, if the total outflow is nonzero (i.e., $\sum_{j \in L} y _{\ell j} >0$), then constraint \eqref{cons:material-balance-ell} implies $\alpha_{\ell k} = \frac{\sum_{i \in I} \lambda_{ik} y_{i\ell}}{\sum_{j \in J}y_{\ell j}}$;
otherwise, by constraints \eqref{cons:flow-conservation} and \eqref{cons:capacity-arc}, $\sum_{i \in I} y _{i\ell} =\sum_{j \in L} y _{\ell j} =0$ must hold and thus no mixing occurs at pool $\ell$.
Thus, constraint \eqref{cons:material-balance-ell} reduces to the trivial equality $0=0$ and  $\alpha_{\ell k}$ can take an arbitrary value.
For simplicity of discussion, if $\sum_{i \in I} y _{i\ell} =\sum_{j \in L} y _{\ell j} =0$ holds for some $\ell \in L$, we assume that $\alpha_{\ell k} = 0$ holds for all $k \in K$ in the following.
Combining the two cases, we can set $\alpha_{\ell k} = q_{\ell k}(y)$ for all $\ell \in L $ and $k \in K$ in formulation \eqref{prob:pooling-P}, where $q_{\ell k} (y)$ is defined in \eqref{def:quality}, and obtain the following equivalent \NLP formulation for the pooling problem:

\begin{equation}\label{prob:pooling-F}
\tag{F}
\max_{y} \left\{ \sum_{(i,j) \in A} w_{ij} y_{ij} :
	\eqref{cons:flow-conservation}, ~
\eqref{cons:capacity-node-i}\text{--}
\eqref{cons:capacity-arc}, ~
\eqref{cons:material-balance-j-lb-Fy},  ~
\eqref{cons:material-balance-j-ub-Fy}  ~
\right\},
\end{equation}
where

	\begin{align}
		& \sum_{i \in I} \lambda_{ik} y_{ij}
		+ \sum_{\ell \in L} q_{\ell k}(y) y_{\ell j} \ge 
		\lambda_{jk}^{\min} \sum_{i \in I \cup L} y_{ij}, \ \forall \ j \in J, \ k \in K,
		\label{cons:material-balance-j-lb-Fy}\\
		& \sum_{i \in I} \lambda_{ik} y_{ij}
		+ \sum_{\ell \in L} q_{\ell k}(y) y_{\ell j} \le
		\lambda_{jk}^{\max} \sum_{i \in I \cup L} y_{ij}, \ \forall \ j \in J, \ k \in K.
		\label{cons:material-balance-j-ub-Fy}
	\end{align}

Observe that only the flow variables $y$ are involved in formulation \eqref{prob:pooling-F} while the quality values are directly reflected through functions $q_{\ell k}(y)$.
Thus, compared with  formulation \eqref{prob:pooling-P}, formulation \eqref{prob:pooling-F} avoids maintaining the relation between the quality variables $\alpha$ and flow variables $y$, thereby  significantly facilitating the algorithmic design.

On the other hand, unlike formulation \eqref{prob:pooling-P} which is a smooth optimization problem, formulation \eqref{prob:pooling-F} is a nonsmooth optimization problem  as for a given point $y^t$, function $q_{\ell k}(y)$ (or $q_{\ell k}(y)y_{\ell j}$) is indifferentiable when $\sum_{r \in J} y_{\ell r}^t = 0$.
However, this is not a big issue when designing an \SLP-type algorithm since it only requires to find a ``good'' linear approximation for the terms  $\{q_{\ell k}(y)y_{\ell j}\}$ at point $y^t$.
In the next two subsections, we will develop two \SLP-type algorithms based on formulation \eqref{prob:pooling-F}.

\subsection{Successive linear programming algorithm based formulation \eqref{prob:pooling-F}} \label{subsect:flow-SLP}

In this subsection, we adapt the \SLP framework to formulation \eqref{prob:pooling-F} and develop a new \SLP algorithm.
We also analyze the relation of the newly proposed \SLP algorithm to the \DR algorithm.

\subsubsection{Proposed algorithm}

In order to design an \SLP algorithm based on formulation \eqref{prob:pooling-F}, let us first consider the linear approximation of $q_{\ell k}(y) y_{\ell j}$ at point $y^t$.
If $\sum_{r \in J} y_{\ell r}^t > 0$, then $q_{\ell k}(y)$ is continuously differentiable at point $y^t$ and thus we can linearize  $q_{\ell k}(y) y_{\ell j}$ using its first-order Taylor series expansion at point $y^t$:
\begin{equation*}
	q_{\ell k}(y) y_{\ell j}
	\approx 
	q_{\ell k}(y^t) y_{\ell j}^t + \sum_{(i,r) \in A} \frac{\partial (q_{\ell k}(y^t) y_{\ell j}^t)}{\partial y_{ir}} (y_{ir} - y_{ir}^t).
\end{equation*}
Otherwise, $q_{\ell k} (y)$ is indifferentiable at point $y^t$, and we decide to approximate $q_{\ell k}(y) y_{\ell j}$ as $q_{\ell k}(y) y_{\ell j} \approx q_{\ell k}(y^t) y_{\ell j}= 0 \times y_{\ell j} = 0$.
Combining the two cases, we obtain 
\begin{equation}\label{tmp-approx3}
	q_{\ell k}(y) y_{\ell j} \approx \tau_{\ell j k}(y): = 
	\left\{
	\begin{aligned}
		& q_{\ell k}(y^t) y_{\ell j}^t + \sum_{(s,r) \in A} \frac{\partial (q_{\ell k}(y^t) y_{\ell j}^t)}{\partial y_{sr}} (y_{sr} - y_{sr}^t), 
		&& \text{if} ~ \sum_{r \in J} y_{\ell r}^t > 0; \\
		& 0, 
		&& \text{otherwise}.
	\end{aligned}
	\right.
\end{equation}
Note that the linear term $\tau_{\ell j k}(y)$ in \eqref{tmp-approx3} depends on the current iterate $y^t$ but we omit this dependence for notations convenience.

By substituting $q_{\ell k}(y) y_{\ell j}$ with the linear terms $\tau_{\ell j k}(y)$ into constraints \eqref{cons:material-balance-j-lb-Fy} and \eqref{cons:material-balance-j-ub-Fy}, we obtain 

	\begin{align}
		& \sum_{i \in I} \lambda_{ik} y_{ij}
		+ \sum_{\ell \in L} \tau_{\ell j k}(y) \ge 
		\lambda_{jk}^{\min} \sum_{i \in I \cup L} y_{ij}, \ \forall \ j \in J, \ k \in K,
		\label{ySLP:material-balance-j-lb}\\
		& \sum_{i \in I} \lambda_{ik} y_{ij}
		+ \sum_{\ell \in L} \tau_{\ell j k}(y) \le
		\lambda_{jk}^{\max} \sum_{i \in I \cup L} y_{ij}, \ \forall \ j \in J, \ k \in K,
		\label{ySLP:material-balance-j-ub}
	\end{align}

and a new \LP approximation of problem \eqref{prob:pooling-F} at point $y^t$:
\begin{equation}
	\label{prob:pooling-F-SLP}
	\tag{$\text{SLP-F}(y^t)$}
	\max_{y} 
	\left\{
	\sum_{(i,j) \in A} w_{ij} y_{ij} : 
		\eqref{cons:flow-conservation}, ~
	\eqref{cons:capacity-node-i}\text{--}
	\eqref{cons:capacity-arc}, ~
	\eqref{ySLP:material-balance-j-lb}, ~
	\eqref{ySLP:material-balance-j-ub}
	\right\}.
\end{equation}
This enables to develop a new \SLP algorithm based on formulation \eqref{prob:pooling-F}; see \cref{alg:SLP-y}.
\begin{algorithm}
	\caption{The successive linear programming algorithm based on formulation \eqref{prob:pooling-F}}
	\KwIn{Choose an initial solution $y^0$ and a maximum number of iterations $t_{\max}$.}
	\label{alg:SLP-y}
	Set $t \leftarrow 0$\;
	\While{$t \le t_{\max}$}{
		Solve \eqref{prob:pooling-F-SLP} to obtain a new iterate $y^{t+1}$\;
		\If{$y^{t+1} = y^t$}{
			\textbf{Stop} with a feasible solution $y^t$ of formulation \eqref{prob:pooling-F}\;
		}
		Set $t \leftarrow t+1$\;
	}
\end{algorithm}

\subsubsection{Relation to the \DR algorithm}

In order to analyze the relation between the \DR algorithm and the proposed \SLP algorithm based on formulation \eqref{prob:pooling-F}, let us first discuss the relation of the \LP subproblems in the two algorithms, i.e., problems \eqref{prob:pooling-P-DR} and  \eqref{prob:pooling-F-SLP}.
Recall that problem \eqref{prob:pooling-P-DR} is developed by first linearizing the \NLP problem \eqref{prob:pooling-P} at point $(\alpha^t, y^t)$, and then projecting variables $\alpha$ out from the \LP approximation problem and refining the resultant problem (i.e., removing some unnecessary constraints in \eqref{SLP:material-balance-ell}).
Problem \eqref{prob:pooling-F-SLP}, however, is developed in a reverse manner. 
It can be obtained by first projecting the $\alpha$ variables out from the \NLP formulation \eqref{prob:pooling-P} and then linearizing the resultant \NLP formulation using the first-order Taylor series expansion technique.
The development of the two \LP approximation problems \eqref{prob:pooling-P-DR} and  \eqref{prob:pooling-F-SLP} is intuitively illustrated in \cref{fig:relationship-nlp-lp}. 
The following result shows, somewhat surprising, that the two \LP  approximation problems \eqref{prob:pooling-P-DR} and  \eqref{prob:pooling-F-SLP} are equivalent (under the trivial assumption that the values $\alpha^t$ are computed using the formula \eqref{def:quality} with $y = y^t$), although they are derived in different ways and they take different forms.

\begin{figure}\centering
	\begin{tikzpicture} [>=stealth,thick,shorten >=0pt,auto,node distance=3.5cm]
		\tikzstyle{every node}=[scale=0.9]
		\tikzstyle{every text node part} = [font=\footnotesize, scale=.9]
		\tikzstyle{arrow}=[->,>=stealth, thick]
		\tikzstyle{arrow2}=[<->,>=stealth, thick]
		\tikzstyle{rectround} = [align=center, rounded corners,inner sep = 2pt,draw,rectangle,line width = 1pt,fill=blue!20, minimum width = 3.5cm, minimum height = 1cm]
		\tikzstyle{rectdash} = [rounded corners,minimum height=5.5cm,draw,rectangle,line width = 1pt,fill = white,opacity=0.3,minimum width = 3cm, style = dashed]
		\tikzstyle{juxing3} = [rounded corners,minimum height = 1.9cm,draw,rectangle,line width = 1pt,fill = white,minimum width = 3.9cm, style = dotted]
		
		\node (probSLP) [rectround] {\LP problem \eqref{prob:pooling-P-SLP}\\ 
			variables: $\alpha$, $y$};
		\node (probDR) [rectround, right=60pt of probSLP] {\LP problem \eqref{prob:pooling-P-DR}\\ 
			variables: $y$};
		
		\node (probF) [rectround, below=30pt of probSLP] {\NLP problem \eqref{prob:pooling-F}\\
			variables: $y$};
		\node (probFSLP) [rectround, below=30pt of probDR] {\LP problem \eqref{prob:pooling-F-SLP}\\
			variables: $y$};
		
		\node (probP) [rectround, left=60pt of $(probSLP)!.5!(probF)$] {\NLP problem \eqref{prob:pooling-P}\\ 
			variables: $\alpha$, $y$};
		
		\path (probP.north)  edge  [bend left=5, arrow]  node [sloped] {Linearization} (probSLP.west);
		\path (probP.south) edge  [bend right=5, arrow]  node [sloped,below] {Projection} (probF.west);
		\path (probSLP.east) edge [arrow] node [sloped, above] {Projection} node [sloped, below] {Refinement} (probDR.west);
		\path (probF.east) edge [arrow] node [sloped, above] {Linearization} (probFSLP.west);
		\path (probDR.east) edge [bend left=45] node [right] {Equivalent} (probFSLP.east);
	\end{tikzpicture}
	\captionsetup{width=\linewidth}
	\caption{Relations of the \LP approximation problems \eqref{prob:pooling-P-DR} and  \eqref{prob:pooling-F-SLP}.}
	\label{fig:relationship-nlp-lp}
\end{figure}

\begin{theorem}\label{thm:relation-DR-SLPy}
	Given $y^t \in \mathbb{R}_+^{|A|}$, let $\alpha^t$ be defined as: 
	\begin{equation}\label{eq:alphaq}
		\alpha^t_{\ell k} = q_{\ell k}(y^t), ~\forall~ \ell \in L, ~k \in K,
	\end{equation}
	where $\{q_{\ell k}(y)\}$ are defined by \eqref{def:quality}.
	Then the linearization of $\alpha_{\ell k} y_{\ell j}$ at point $(\alpha^t, y^t)$ in problem \eqref{prob:pooling-P-DR} is equivalent to that of $q_{\ell k}(y) y_{\ell j}$ at point $y^t$ in problem \eqref{prob:pooling-F-SLP}; that is,  
	for any given $y \in \mathbb{R}_+^{|A|}$, the following equations hold
	\begin{equation}\label{eq:sigmatau}
		\sigma_{\ell j k}(y) = \tau_{\ell j k}(y),~\forall~\ell \in L, ~ j \in J, ~k \in K,
	\end{equation}
	where $\{\sigma_{\ell j k}(y)\}$ and $\{\tau_{\ell j k}(y)\}$ are defined by 
	\eqref{sigmaapprox}
	and \eqref{tmp-approx3}, respectively.
	Moreover, the two \LP approximation problems  \eqref{prob:pooling-P-DR} and  \eqref{prob:pooling-F-SLP} are equivalent. 
\end{theorem}
\begin{proof}
	If $\sum_{r \in J} y_{\ell r}^t = 0$, then $\sigma_{\ell j k} (y) = 0 = \tau_{\ell j k}(y)$ follows directly from the definitions in  \eqref{sigmaapprox} and \eqref{tmp-approx3}.
	Otherwise, $\sum_{r \in J} y_{\ell r}^t > 0$ holds.
	From the definition of $\tau_{\ell j k}(y)$ in \eqref{tmp-approx3}, it follows that
	\begin{equation}\label{tmpeq0}
		\begin{aligned}
			\tau_{\ell j k}(y) 
			& = q_{\ell k}(y^t) y^t_{\ell j} + \sum_{(s,r) \in A} \frac{\partial (q_{\ell k}(y^t) y_{\ell j}^t)}{\partial y_{sr}} (y_{sr} - y_{sr}^t)\\
			& = q_{\ell k}(y^t) y^t_{\ell j} + \sum_{(s,r) \in A} \left(\frac{\partial q_{\ell k}(y^t)}{\partial y_{sr}} y_{\ell j}^t + 
			\frac{\partial y_{\ell j}^t}{\partial y_{sr}} q_{\ell k}(y^t)\right)(y_{sr} - y_{sr}^t)\\
			&  \overset{\text{(a)}}{=} q_{\ell k}(y^t) y^t_{\ell j} + \sum_{(s,r) \in A} \frac{\partial q_{\ell k}(y^t)}{\partial y_{sr}} y_{\ell j}^t (y_{sr} - y_{sr}^t)+ 
			q_{\ell k}(y^t) (y_{\ell j} - y_{\ell j}^t)\\
			& = 	q_{\ell k}(y^t) y_{\ell j}+ y_{\ell j}^t\sum_{(s,r) \in A} \frac{\partial q_{\ell k}(y^t)}{\partial y_{sr}}  (y_{sr} - y_{sr}^t),
		\end{aligned}
	\end{equation}
	where (a) follows from $\frac{\partial y_{\ell j}}{\partial y_{sr}} = 1$ if  $s = \ell$ and $r = j$, and $\frac{\partial y_{\ell j}}{\partial y_{sr}}  =0 $ otherwise.
	By \eqref{def:quality}, we have 
	\begin{equation}\label{tmp-derivative-F}
			\frac{\partial q_{\ell k}(y^t)}{\partial y_{sr}}=
		\left\{ \ 
		\begin{aligned}
			& \frac{\lambda_{sk}}{\sum_{j \in J} y_{\ell j}^t},
			&& \text{if}~s \in I \text{ and } r = \ell; \\
			& - \frac{\sum_{i \in I} \lambda_{i k} y_{i \ell}^t}{\left(\sum_{j \in J} y_{\ell j}^t\right)^2} = - \frac{q_{\ell k}(y^t)}{\sum_{j \in J} y_{\ell j}^t} ,
			&& \text{if}~s = \ell~\text{and}~r \in J; \\
			& 0, 
			&& \text{otherwise},
		\end{aligned}
		\right. 
		 ~~~\forall \ (s,r) \in A.
	\end{equation}
	Thus,
	\begin{equation}\label{tmpeq1}
		\begin{aligned}
			& y_{\ell j}^t \sum_{(s,r) \in A} \frac{\partial q_{\ell k}(y^t)}{\partial y_{sr}} (y_{sr} - y_{sr}^t)\\
			& = y_{\ell j}^t \left(\sum_{(i,\ell) \in I \times \{\ell\} } \frac{\lambda_{ik}}{\sum_{j' \in J} y_{\ell j'}^t} (y_{i\ell} - y_{i\ell}^t) 
			- \sum_{(\ell,r) \in \{\ell\} \times J} \frac{q_{\ell k}(y^t)}{\sum_{j' \in J} y_{\ell j'}^t}  (y_{\ell r} - y_{\ell r}^t)\right)\\
			& = \frac{y_{\ell j}^t}{\sum_{r \in J} y_{\ell r}^t} 			
			\left( 
			\sum_{i \in I} \lambda_{ik} (y_{i \ell} - y_{i \ell}^t)
			- \sum_{r \in J} q_{\ell k}(y^t) (y_{\ell r} - y_{\ell r}^t)
			\right)\\
			&  \overset{\text{(a)}}{=} \frac{y_{\ell j}^t}{\sum_{r \in J} y_{\ell r}^t} 			
			\left( 
			\sum_{i \in I} \lambda_{ik} y_{i \ell} 
			- \sum_{r \in J} q_{\ell k}(y^t) y_{\ell r}
			\right),
		\end{aligned}
	\end{equation}
	where (a) follows from $q_{\ell k} (y^t)=  \frac{\sum_{i \in I} \lambda_{ik} y^t_{i\ell}}{\sum_{j \in J} y^t_{\ell j}}$, or equivalently, $q_{\ell k} (y^t)\sum_{j \in J} y^t_{\ell j}= \sum_{i \in I} \lambda_{ik} y^t_{i\ell}$.
	Combining \eqref{eq:alphaq}, \eqref{tmpeq0}, and  \eqref{tmpeq1}, we obtain 
	\begin{equation}
			\tau_{\ell j k}(y)  = \alpha_{\ell k}^t y_{\ell j} + \frac{y_{\ell j}^t}{\sum_{r \in J} y_{\ell r}^t}   \left(\sum_{i \in I} \lambda_{ik} y_{i \ell} - \alpha_{\ell k}^t \sum_{r \in J} y_{\ell r}\right) = \sigma_{\ell j k}(y).
	\end{equation}
	This completes the proof.
\end{proof}

\cref{thm:relation-DR-SLPy} immediately implies that the newly proposed \SLP algorithm in \cref{alg:SLP-y} and the \DR algorithm in \cref{alg:DR-pooling-P} are also \emph{equivalent}.
Specifically, if we choose the same initial point $y^0$, the two algorithms will converge to same solution for the pooling problem.
Thus, the \DR algorithm can be interpreted  as the \SLP algorithm based on formulation \eqref{prob:pooling-F}.  
This new perspective of the \DR algorithm enables to develop more sophisticated algorithms based on formulation \eqref{prob:pooling-F} to find high-quality solutions for the pooling problem. 
In the next subsection, we will develop a penalty \SLP algorithm based on formulation \eqref{prob:pooling-F}. 

\subsection{Penalty distributed recursion algorithm}\label{subsect:flow-PSLP}

The \SLP algorithm based on formulation \eqref{prob:pooling-F} (or the \DR algorithm) solves an \LP approximation subproblem \eqref{prob:pooling-F-SLP} where the nonlinear constraints in formulation \eqref{prob:pooling-F} are linearized at a point $y^t$. 
Such a linearization \eqref{prob:pooling-F-SLP} may return a new iterate $y^{t+1}$ which is infeasible to the original formulation \eqref{prob:pooling-F} (as the nonlinear constraints \eqref{cons:material-balance-j-lb-Fy} and \eqref{cons:material-balance-j-ub-Fy} may be violated at $y^{t+1}$), 
even if the former iterate $y^t$ is a feasible solution of formulation \eqref{prob:pooling-F}.
The goal of this subsection is to develop an improved  algorithm  which better takes the feasibility of the original formulation \eqref{prob:pooling-F} into consideration during the iteration procedure. 
In particular, we integrate the penalty algorithmic framework \cite{Nocedal1999}
into the \SLP algorithm to solve formulation \eqref{prob:pooling-F}, where the violations of the (linearized) nonlinear constraints \eqref{cons:material-balance-j-lb-Fy} and \eqref{cons:material-balance-j-ub-Fy} are penalized in the objective function of the \LP approximation problem with the penalty terms increasing when the constraint violations tend to be large.
It should be mentioned that the penalty \SLP algorithms based on formulation \eqref{prob:pooling-P} (that involves the $\alpha$ and $y$ variables) have been investigated in the literature; see, e.g., \citet{Zhang1985,Baker1985}.

To develop the penalty \SLP algorithm based on formulation \eqref{prob:pooling-F}, let us first consider the following  penalty problem:  
\begin{equation}\label{prob:pooling-F-penalty}
	\tag{$\text{F}^p$}
	\begin{aligned}
		& \max_{y, \, s} 
		\left\{
		\sum_{(i,j) \in A} w_{ij} y_{ij} - \sum_{j \in J} \sum_{k \in K} (\mu_{jk} s_{jk}^{\min}+ \nu_{jk} s_{jk}^{\max}) :
		\right.\\
		&
		\quad 
		\left. \phantom{\sum_{(i,j) \in A} c_{ij} y_{ij} }
		\eqref{cons:flow-conservation},~
		\eqref{cons:capacity-node-i}\text{--}
		\eqref{cons:capacity-arc}, ~
		\eqref{cons:material-balance-j-lb-Fy-penalty}, ~
		\eqref{cons:material-balance-j-ub-Fy-penalty}, ~
		s_{jk}^{\max}, ~s_{jk}^{\min} \ge 0, \ \forall \ j \in J, \ k \in K
		\right\},
	\end{aligned}
\end{equation}
where
\begin{align}
	& \sum_{i \in I} \lambda_{ik} y_{ij}
	+ \sum_{\ell \in L} q_{\ell k}(y)y_{\ell j}
	+ s_{jk}^{\min}
	\ge 
	\lambda_{jk}^{\min} \sum_{i \in I \cup L} y_{ij} 
	, \ \forall \ j \in J, \ k \in K,
	\label{cons:material-balance-j-lb-Fy-penalty}\\
	& \sum_{i \in I} \lambda_{ik} y_{ij}
	+ \sum_{\ell \in L} q_{\ell k}(y)y_{\ell j}
	- s_{jk}^{\max}
	\le
	\lambda_{jk}^{\max} \sum_{i \in I \cup L} y_{ij}, \ \forall \ j \in J, \ k \in K,
	\label{cons:material-balance-j-ub-Fy-penalty}
\end{align}
the nonnegative variables 
$s_{jk}^{\min}$ and $s_{jk}^{\max}$ are slack variables characterizing the infeasibilities/violations of constraints \eqref{cons:material-balance-j-lb-Fy} and \eqref{cons:material-balance-j-ub-Fy}, respectively, and 
$\mu_{jk}>0$ and $v_{jk}>0$ are the penalty parameters. 
Observe that the penalty problem \eqref{prob:pooling-F-penalty} is equivalent to the penalty version of problem \eqref{prob:pooling-P}:
\begin{equation}\label{prob:pooling-P-penalty}
	\tag{$\text{P}^p$}
	\begin{aligned}
		& \max_{y, \,\alpha,\, s} 
		\left\{
		\sum_{(i,j) \in A} w_{ij} y_{ij} - \sum_{j \in J} \sum_{k \in K} (\mu_{jk} s_{jk}^{\min}+ \nu_{jk} s_{jk}^{\max}) :
		\right.\\
		&
		\quad 
		\left. \phantom{\sum_{(i,j) \in A} c_{ij} y_{ij} }
		\eqref{cons:flow-conservation}, ~
		\eqref{cons:material-balance-ell},~
		\eqref{cons:capacity-node-i}\text{--}
		\eqref{cons:capacity-arc}, ~
		\eqref{cons:material-balance-j-lbP}, ~
		\eqref{cons:material-balance-j-ubP}, ~
		s_{jk}^{\max}, ~s_{jk}^{\min} \ge 0, \ \forall \ j \in J, \ k \in K
		\right\},
	\end{aligned}
\end{equation}
where 
	\begin{align}
	& \sum_{i \in I} \lambda_{ik} y_{ij}
	+ \sum_{\ell \in L} \alpha_{\ell k} y_{\ell j} + s_{jk}^{\min} \ge 
	\lambda_{jk}^{\min} \sum_{i \in I \cup L} y_{ij}, \ \forall \ j \in J, \ k \in K,
	\label{cons:material-balance-j-lbP}\\
	& \sum_{i \in I} \lambda_{ik} y_{ij}
	+ \sum_{\ell \in L} \alpha_{\ell k} y_{\ell j} - s_{jk}^{\max}\le
	\lambda_{jk}^{\max} \sum_{i \in I \cup L} y_{ij}, \ \forall \ j \in J, \ k \in K.
	\label{cons:material-balance-j-ubP}
\end{align}
Under mild conditions,  a local minimum of problem \eqref{prob:pooling-P} is also a local minimum of problem \eqref{prob:pooling-P-penalty}, provided that $\mu_{jk}$ and $\nu_{jk}$ are larger than the optimal Lagrange multipliers for the nonlinear constraints \eqref{cons:material-balance-j-lb} and \eqref{cons:material-balance-j-ub}; see \citet[Chapter 13, Exact Penalty Theorem]{Luenberger2021}.
This, together with the equivalence of problems \eqref{prob:pooling-F-penalty} and \eqref{prob:pooling-P-penalty} and the equivalence of problems \eqref{prob:pooling-F} and \eqref{prob:pooling-P}, motivates us to find a feasible solution for the pooling problem by solving the penalty problem \eqref{prob:pooling-F-penalty}.

The penalty problem \eqref{prob:pooling-F-penalty} can be solved by the \SLP algorithm where problem \eqref{prob:pooling-F-penalty} is still approximated by an \LP problem at a point $y^t$ and the penalty parameters are dynamically updated, as to consider the feasibility of the future iterates. 
To this end, similar to problem \eqref{prob:pooling-F}, the nonlinear terms $\{q_{\ell k}(y)y_{\ell j}\}$ are linearized by the linear functions $\{\tau_{\ell j k} (y)\}$ (defined in \eqref{tmp-approx3}), and the penalty problem \eqref{prob:pooling-F-penalty}  is approximated by the following \LP problem at point $y^t$:

\begin{equation}\label{prob:pooling-F-PSLP}
	\tag{$\text{PSLP}(y^t)$}
	\begin{aligned}
		& \max_{y, \, s} 
		\left\{
		\sum_{(i,j) \in A} w_{ij} y_{ij} - \sum_{j \in J} \sum_{k \in K} (\mu_{jk} s_{jk}^{\min}+ \nu_{jk} s_{jk}^{\max}) : 
		\right.\\
		&
		\qquad 
		\left. \phantom{\sum_{(i,j) \in A} c_{ij} y_{ij} }
		\eqref{cons:flow-conservation}, ~
		\eqref{cons:capacity-node-i}\text{--}
		\eqref{cons:capacity-arc}, ~
		\eqref{yPSLP:material-balance-j-lb}, ~
		\eqref{yPSLP:material-balance-j-ub}, ~
		s_{jk}^{\min}, s_{jk}^{\max} \ge 0, \ \forall \ j \in J, \ k \in K
		\right\},
	\end{aligned}
\end{equation}
where
	\begin{align}
		& \sum_{i \in I} \lambda_{ik} y_{ij}
		+ \sum_{\ell \in L} \tau_{\ell j k}(y) 
		+ s_{jk}^{\min}
		\ge 
		\lambda_{jk}^{\min} \sum_{i \in I \cup L} y_{ij} 
		, \ \forall \ j \in J, \ k \in K,
		\label{yPSLP:material-balance-j-lb}\\
		& \sum_{i \in I} \lambda_{ik} y_{ij}
		+ \sum_{\ell \in L} \tau_{\ell j k}(y) 
		- s_{jk}^{\max}
		\le
		\lambda_{jk}^{\max} \sum_{i \in I \cup L} y_{ij}, \ \forall \ j \in J, \ k \in K.
		\label{yPSLP:material-balance-j-ub}
	\end{align}
Solving the \LP approximation problem \eqref{prob:pooling-F-PSLP}, we obtain a new iterate $y^{t+1}$.
Now, for each $j \in J$ and $k \in K$, if the nonlinear constraint \eqref{cons:material-balance-j-lb-Fy} or \eqref{cons:material-balance-j-ub-Fy} is  violated by the new iterate $y^{t+1}$, 
we increase the penalty parameter $\mu_{jk}$ or  $\nu_{jk}$ by a factor of $\delta > 1$ (as to force the future iterates to be feasible); otherwise, the current $\mu_{jk}$ or  $\nu_{jk}$ is enough to force the corresponding nonlinear constraint to be satisfied and does not need to be updated.  
The overall penalty \SLP algorithm is summarized in \cref{alg:PSLP-y}.

\begin{algorithm}
	\caption{The penalty distributed recursion algorithm}
	\label{alg:PSLP-y}
	\KwIn{Choose an initial solution $y^0$, positive penalty parameters $\mu>0$ and $\nu > 0$, a constant $\delta>1$, and a maximum number of iterations $t_{\max}$.}
	Set $t \leftarrow 0$\;
	\While{$t \le t_{\max}$}{
		Solve the \LP problem \eqref{prob:pooling-F-PSLP} to obtain the solution $(y^{t+1}, s^{t+1})$\;
		\If{$s^{t+1}=0$ and $y^{t+1} = y^t$}{
			\textbf{Stop} with a feasible solution $y^t$ of formulation \eqref{prob:pooling-F}\;
		}
		\For{$j \in J$ and $k \in K$}{
			\lIf{${[s^{t+1}]}_{jk}^{\max} > 0$}{set $\mu_{jk} \leftarrow \delta \mu_{jk}$}
			\lIf{$[s^{t+1}]_{jk}^{\min} > 0$}{set $\nu_{jk} \leftarrow \delta \nu_{jk}$}
		}
		Set $t \leftarrow t + 1$\;
	}	
	\KwRet{$y^{t+1}$}\;
\end{algorithm}

Two remarks on the proposed \cref{alg:PSLP-y} are as follows.
First, \cref{alg:PSLP-y} is a variant of \cref{alg:SLP-y} that uses penalty terms to render the feasibility of the iterates. 
This, together with the equivalence of \cref{alg:SLP-y} and the \DR algorithm, indicates that \cref{alg:PSLP-y} can be interpreted as a penalty version of the \DR algorithm (thus called \PDR algorithm) where the violations of constraints \eqref{DR:material-balance-j-lb} and \eqref{DR:material-balance-j-ub} are penalized in the objective function of \eqref{prob:pooling-F-PSLP}.

Second, for a fixed point $y^t$, the \LP approximation problem  \eqref{prob:pooling-F-SLP} in \cref{alg:SLP-y}  can be seen as a restriction of the \LP approximation problem \eqref{prob:pooling-F-PSLP} in \cref{alg:PSLP-y}  where the slack variables  $s_{jk}^{\min}$ and $s_{jk}^{\max}$ in \eqref{prob:pooling-F-PSLP} are all set to zero.
Therefore, solving problem \eqref{prob:pooling-F-PSLP} can return a solution that has a better objective value than that of   the optimal solution of problem \eqref{prob:pooling-F-SLP}.
As such, the proposed \PDR algorithm can construct a sequence of iterates (i) for which all linear constraints \eqref{cons:flow-conservation} and 
\eqref{cons:capacity-node-i}--\eqref{cons:capacity-arc} are satisfied and (ii) whose objective values are generally larger than those of the iterates constructed by the classic \DR algorithm.
With this favorable feature, the proposed \PDR algorithm is more likely to return a better feasible solution than that returned by the classic \DR algorithm. 
In \cref{sect:compu}, we will further present computational results to verify this.
\section{Computational results}\label{sect:compu}

In this section, we present the computational results to demonstrate the effectiveness of the proposed \PDR algorithm based on formulation \eqref{prob:pooling-F} (i.e., \cref{alg:PSLP-y}) over the \SLP and \DR algorithms based on formulation \eqref{prob:pooling-P} (i.e., \cref{alg:SLP0,alg:DR-pooling-P}) for the pooling problem.
The three algorithms were implemented in Julia 1.9.2 using \gurobi 11.0.1 as the \LP solver. 
The computations were conducted on a cluster of Intel(R) Xeon(R) Gold 6230R CPU @ 2.10 GHz computers.
We solve problem \eqref{mcf} to obtain a point $y^0$ to initialize the three algorithms.

In addition, we set the maximum number of iterations $t_{\max} = 100$ for all the three algorithms.
For the proposed \PDR algorithm, we set parameters $\mu_{jk} = \nu_{jk} = 1$ for all $j \in J$ and $k \in K$ and $\delta = 10$.

\subsection{Computational results on benchmark instances in the literature}

We first compare the performance of the \SLP, \DR, and \PDR algorithms (denoted as settings \testSLP, \testDR, and \testPDR) on the $10$ benchmark instances taken from the literature:
the \texttt{AST1}, \texttt{AST2}, \texttt{AST3}, and \texttt{AST4} instances from \citet{Adhya1999},
the \texttt{BT4} and \texttt{BT5} instances from  \citet{BenTal1994},
the \texttt{F2} instance from \citet{Foulds1992},  and
the \texttt{H1}, \texttt{H2}, and \texttt{H3} instances in \citet{Haverly1978}.

\begin{table}[htbp]
	\centering
	\caption{Computational results on $10$ benchmark instances under settings \testSLP, \testDR, and \testPDR.}
	\tabcolsep=2pt
	\label{table:SLP-YSLP-PYSLP}
	\renewcommand{\arraystretch}{1.2}
	\footnotesize
	\resizebox{\linewidth}{!}{
		\begin{tabular}{@{}lrrrrrrrrrrrrrrrrc@{}}
	\toprule
	\multicolumn{1}{c}{\multirow{2}{*}{\tblid}}
	&\multicolumn{5}{c}{\testSLP}
	&\multicolumn{5}{c}{\testDR}
	&\multicolumn{5}{c}{\testPDR}
	&\multicolumn{1}{c}{\multirow{2}{*}{\tblBV}}
	\\
	\cmidrule(r){2-6}\cmidrule(r){7-11}\cmidrule(r){12-16}
	
	&\multicolumn{1}{c}{\tblT}&\multicolumn{1}{c}{\tblObj}&\multicolumn{1}{c}{\tblIter}&\multicolumn{1}{c}{\tblGap}&\multicolumn{1}{c}{\tblObjAver}
	&\multicolumn{1}{c}{\tblT}&\multicolumn{1}{c}{\tblObj}&\multicolumn{1}{c}{\tblIter}&\multicolumn{1}{c}{\tblGap}&\multicolumn{1}{c}{\tblObjAver}
	&\multicolumn{1}{c}{\tblT}&\multicolumn{1}{c}{\tblObj}&\multicolumn{1}{c}{\tblIter}&\multicolumn{1}{c}{\tblGap}&\multicolumn{1}{c}{\tblObjAver}
	& 
	\\\midrule
	AST1-5-2-4-4        	& $<0.01$	&    0.00	&       2	&   100.0	&     0.0	& $<0.01$	&    0.00	&      12	&   100.0	&    50.1	& $<0.01$	&  340.93	&      10	&    38.0	&    85.8	&  549.80\\ 
	AST2-5-2-4-6        	& $<0.01$	&    0.00	&       2	&   100.0	&     0.0	& $<0.01$	&    0.00	&       8	&   100.0	&    49.9	& $<0.01$	&  509.78	&      14	&     7.3	&    82.8	&  549.80\\ 
	AST3-8-3-4-6        	& $<0.01$	&    0.00	&       2	&   100.0	&     0.0	& $<0.01$	&  561.04	&      12	&     0.0	&    52.9	&    0.05	&  531.05	&     100	&     5.3	&    91.6	&  561.04\\ 
	AST4-8-2-5-4        	& $<0.01$	&  105.00	&       2	&    88.0	&     9.6	& $<0.01$	&  470.83	&       7	&    46.4	&    84.5	& $<0.01$	&  877.65	&      10	&     0.0	&    98.7	&  877.65\\ 
	BT4-4-1-2-1         	& $<0.01$	&  350.00	&       5	&    22.2	&    17.0	& $<0.01$	&  450.00	&       4	&     0.0	&    23.8	& $<0.01$	&  450.00	&       5	&     0.0	&    43.5	&  450.00\\ 
	BT5-5-3-5-2         	&    0.03	& 2865.19	&      37	&    18.1	&    46.0	& $<0.01$	& 3500.00	&       7	&     0.0	&    63.5	&    0.01	& 3500.00	&      18	&     0.0	&    63.0	& 3500.00\\ 
	F2-6-2-4-1          	& $<0.01$	&  600.00	&       4	&    45.5	&    16.2	& $<0.01$	& 1100.00	&       6	&     0.0	&    36.4	& $<0.01$	& 1100.00	&       6	&     0.0	&    56.0	& 1100.00\\ 
	H1-3-1-2-1          	& $<0.01$	&  300.00	&       5	&    25.0	&    14.6	& $<0.01$	&  400.00	&       4	&     0.0	&    21.4	& $<0.01$	&  400.00	&       5	&     0.0	&    41.7	&  400.00\\ 
	H2-3-1-2-1          	& $<0.01$	&  300.00	&       3	&    50.0	&    13.0	& $<0.01$	&  600.00	&       3	&     0.0	&    18.5	& $<0.01$	&  600.00	&       4	&     0.0	&    46.3	&  600.00\\ 
	H3-3-1-2-1          	& $<0.01$	&  750.00	&       5	&     0.0	&    32.7	& $<0.01$	&  750.00	&       5	&     0.0	&    35.9	& $<0.01$	&  750.00	&       6	&     0.0	&    48.7	&  750.00\\ 
	\hline
	\multicolumn{1}{l}{\tblAve}
	&    0.00	&        	&     6.7	&    54.9	&    14.9	&    0.00	&        	&     6.8	&    24.6	&    43.7	&    0.01	&        	&    17.8	&     5.1	&    65.8	&        \\ 
	\bottomrule
\end{tabular}

	}
\end{table}

\cref{table:SLP-YSLP-PYSLP} summarizes the performance results of the three settings \testSLP, \testDR, and \testPDR.
In \cref{table:SLP-YSLP-PYSLP}, we report for each instance the instance id, the numbers of inputs, pools, outputs, and quality attributes (combined in column ``\tblid'') and the optimal value obtained in the literature (column ``\tblBV''), which can be found in, e.g., \citet{Audet2004}.
Moreover, we report, under each setting, 
the total CPU time in seconds (\tblT), 
the objective value of the returned feasible solution (\tblObj), 
the number of iterations (\tblIter), and 
the optimality gap of objective values (\tblGap).
The optimality gap \tblGap is defined by $\frac{o^\ast - o}{o^\ast} \times 100\%$, where $o^\ast$ denotes the optimal value (\tblBV) and $o \in \{o_{\testSLP}, o_{\testDR}, o_{\testPDR}\}$ denotes the objective value returned by the corresponding setting.
To gain more insights of the performance of the three algorithms, we also report the average objective ratio of $\frac{o_t}{o_0}$ under column \tblObjAver, where $o_0$ is the optimal value of problem \eqref{mcf} and $o_t$ is the objective value at iterate $t$ (i.e., 
the optimal value of problem \eqref{prob:pooling-P-SLP}, \eqref{prob:pooling-P-DR}, or \eqref{prob:pooling-F-SLP}).
The average objective ratio $\frac{o_t}{o_0} \in [0,1]$ reflects the objective values of the iterates computed by the algorithms: the larger the $\frac{o(y^t)}{o(y^0)}$, the better (as the algorithm is more likely to construct a feasible solution of problem \eqref{prob:pooling-P} or problem \eqref{prob:pooling-F} with a larger objective value).
At the end of the table, we also report the average CPU time, the average relative gap, and the average ratio $\frac{o_t}{o_0}$.

From the results in \cref{table:SLP-YSLP-PYSLP}, we can conclude that \testPDR performs the best in terms of finding high-quality solutions, followed by \testDR and then \testSLP. 
More specifically, compared with \testSLP which finds an optimal solution for only a single instance, \testDR can find an optimal solution for $7$ instances.
This confirms that \eqref{prob:pooling-P-DR} is indeed a better \LP approximation than \eqref{prob:pooling-P-SLP} around the iterate $y^t$.  
Indeed,
(i) 
\DR uses a more accurate solution $\alpha^{t+1}$ (in terms of guaranteeing the strong relations between the $\alpha$ and $y$ variables) for constructing the next \LP subproblem $(\text{DR}(\alpha^{t+1}, y^{t+1}))$; 
and (ii) compared with the \LP subproblem \eqref{prob:pooling-P-SLP} in the \SLP algorithm, the \LP subproblem \eqref{prob:pooling-P-DR} in the \DR algorithm avoids the addition of the unnecessary constraints in \eqref{SLP:material-balance-ell}, which enlarges the feasible region, thereby enabling to return a solution with a larger objective value (see columns \tblObjAver).
Built upon these advantages of \testDR, the proposed \testPDR performs much better than \testSLP; its overall performance is even better than \testDR. 
In particular, \testPDR achieves relative gaps of $38.0\%$, $7.3\%$, and $0.0\%$ for instances \texttt{AST1}, \texttt{AST2}, and \texttt{AST4}, respectively, while those for \testDR are $100.0\%$, $100.0\%$, and $46.4\%$, respectively.
Indeed, for instances \texttt{AST1} and \texttt{AST2}, \testDR was only able to find the all-zero trivial solution (with a zero objective value).
This can be attributed to the reason that the average ratios under setting \testPDR are generally higher than those under  \testDR, which makes the proposed \testPDR more likely to return a feasible solution with a larger objective value; see \cref{table:SLP-YSLP-PYSLP}.

\subsection{Computational results on randomly generated instances}

In order to gain more insights of the computational performance of the \SLP, \DR, and \PDR algorithms, we perform computational experiments on a set of randomly generated instances.
The instances were constructed using a similar procedure as in \citet{Alfaki2013a} and can be categorized into 5 distinct groups labeled as A, B, C, D, and E, each comprising 10 instances.

The number of inputs, pools, outputs, and quality attributes, denoted as $(|I|, |L|, |J|, |K|)$, for groups A--E, are set to $(3,2,3,2)$, $(5,4,3,3)$, $(8,6,6,4)$, $(12,10,8,5)$, and $(10,10,15,12)$, respectively.
Each pair $(i,j) \in I \times (L \cup J)$ has a probability of $0.5$ to appear in the set of arcs $A$ and all pairs in $L \times J$ are recognized as arcs.

The weights $\{w_{ij}\}$ on arcs are calculated as the difference between the unit cost $c_i$ of purchasing raw materials at the input node $i \in I$ and the unit revenue $c_j$ of selling products at output nodes $j \in J$, i.e., $w_{ij} = c_j - c_i$ for all $(i,j) \in A$.
Here, $c_i$, $i \in I$, and $c_j$, $j \in J$, are uniformly chosen from the $\{0, \dots, 5\}$ and $\{5, \dots, 14\}$, respectively, and $c_{\ell}$, $\ell \in L$, are all set to zeros.

The maximum product demands $u_j$ are randomly chosen from $\{20, \dots, 59\}$; the available raw material supplies $u_i$, $i \in I$, and pool capacities $u_{\ell}$, $\ell \in L$, are set to infinity; 
the maximum flows that can be carried on arcs $(i,j) \in A$ are also set  to infinity, i.e., $u_{ij} = + \infty$.
The qualities of inputs $\lambda_{ik}$, $i \in I$, $k \in K$, are randomly selected from $\{0, \dots, 9\}$.
The upper bounds on attribute qualities at output nodes $\lambda^{\max}_{jk}$, $j \in J$, $k \in K$, are randomly chosen from $\{2, \dots, 6\}$ 
and the lower bounds on attribute qualities at output nodes $\lambda^{\min}_{jk}$, $j \in J$, $k \in K$, are set to zeros.

Detailed performance results on randomly generated instances for  \testSLP, \testDR, and \testPDR are shown in \cref{table:SLP-YSLP-PYSLP-random}.

To evaluate the optimality gap of the feasible solution returned by the three settings, we leverage the global solver \gurobi to compute an optimal solution (or best incumbent) of the problem. 
Specifically, we first apply \gurobi to solve \eqref{prob:pooling-P} (within a time limit of 600 seconds) to obtain a feasible solution with the objective value $o_{\testGRB}$, and then take $o^\ast= \max\{o_{\testSLP}, o_{\testDR}, o_{\testPDR}, o_{\testGRB}\}$
as the best objective value for the computation of optimality gap $\frac{o^\ast - o}{o^\ast} \times 100\%$ (of the feasible solution returned by the three settings).
We use ``--'' (under columns \tblObj and \tblGap) to indicate that no feasible solution is found within the maximum number of iterations.
In \cref{fig:gap}, we further plot the performance profiles of optimality gaps (\tblGap) under the settings \testSLP, \testDR, \testPDR, and \testGRB.

\begin{table}[htbp]
	\centering
	\caption{Computational results on randomly generated instances under settings \testSLP, \testDR, \testPDR, and \testGRB.}
	\label{table:SLP-YSLP-PYSLP-random}
	\tabcolsep=2pt
	\renewcommand{\arraystretch}{1.2}
	\scriptsize
	\resizebox{\linewidth}{!}{
		\begin{tabular}{@{}lrrrrrrrrrrrrrrrrrrc@{}}
	\toprule
	\multicolumn{1}{c}{\multirow{2}{*}{\tblid}}
	&\multicolumn{5}{c}{\testSLP}
	&\multicolumn{5}{c}{\testDR}
	&\multicolumn{5}{c}{\testPDR}
	&\multicolumn{3}{c}{\testGRB}
	\\
	\cmidrule(r){2-6}\cmidrule(r){7-11}\cmidrule(r){12-16}\cmidrule(r){17-19}
	
	&\multicolumn{1}{c}{\tblT}&\multicolumn{1}{c}{\tblObj}&\multicolumn{1}{c}{\tblIter}&\multicolumn{1}{c}{\tblGap}&\multicolumn{1}{c}{\tblObjAver}
	&\multicolumn{1}{c}{\tblT}&\multicolumn{1}{c}{\tblObj}&\multicolumn{1}{c}{\tblIter}&\multicolumn{1}{c}{\tblGap}&\multicolumn{1}{c}{\tblObjAver}
	&\multicolumn{1}{c}{\tblT}&\multicolumn{1}{c}{\tblObj}&\multicolumn{1}{c}{\tblIter}&\multicolumn{1}{c}{\tblGap}&\multicolumn{1}{c}{\tblObjAver}
	&\multicolumn{1}{c}{\tblT}&\multicolumn{1}{c}{\tblObj}&\multicolumn{1}{c}{\tblGap}
	\\\midrule
	A1-3-2-3-2          	& $<0.01$	&  160.71	&       2	&    73.8	&    13.5	& $<0.01$	&  612.94	&       3	&     0.0	&    71.6	& $<0.01$	&  612.94	&       4	&     0.0	&    91.4	&    0.08	&  612.94	&     0.0\\ 
	A2-3-2-3-2          	& $<0.01$	&  250.00	&       2	&    43.8	&    34.2	& $<0.01$	&  250.00	&      10	&    43.8	&    47.6	& $<0.01$	&  415.00	&       7	&     6.7	&    47.6	&   61.42	&  445.00	&     0.0\\ 
	A3-3-2-3-2          	& $<0.01$	&  236.62	&       2	&     3.7	&    58.7	& $<0.01$	&  245.73	&       4	&     0.0	&    72.6	&    0.04	&  245.73	&     100	&     0.0	&    72.6	&    0.29	&  245.73	&     0.0\\ 
	A4-3-2-3-2          	& $<0.01$	&  403.10	&       2	&    36.3	&    54.5	& $<0.01$	&  632.70	&       7	&     0.0	&    91.2	& $<0.01$	&  632.70	&       5	&     0.0	&    92.9	&    0.38	&  632.70	&     0.0\\ 
	A5-3-2-3-2          	& $<0.01$	&    0.00	&       2	&   100.0	&     0.0	& $<0.01$	&    0.00	&       3	&   100.0	&    50.0	& $<0.01$	&  280.29	&       4	&     0.0	&    96.2	&    0.15	&  280.29	&     0.0\\ 
	A6-3-2-3-2          	& $<0.01$	&  530.84	&       3	&     0.0	&    51.8	&    0.01	&  530.84	&      11	&     0.0	&    60.4	& $<0.01$	&  510.32	&       4	&     3.9	&    71.0	&    0.11	&  530.84	&     0.0\\ 
	A7-3-2-3-2          	& $<0.01$	&    0.00	&       2	&   100.0	&     0.0	& $<0.01$	&    0.00	&       2	&   100.0	&     0.0	& $<0.01$	&  341.00	&       4	&     0.0	&    72.0	&    0.20	&  341.00	&     0.0\\ 
	A8-3-2-3-2          	& $<0.01$	& 1536.00	&       2	&     9.9	&    86.3	& $<0.01$	& 1704.00	&       4	&     0.0	&    97.0	& $<0.01$	& 1704.00	&       4	&     0.0	&    97.0	&    2.44	& 1704.00	&     0.0\\ 
	A9-3-2-3-2          	& $<0.01$	&  994.50	&       4	&     5.5	&    87.1	& $<0.01$	& 1052.50	&       4	&     0.0	&    92.4	& $<0.01$	& 1052.50	&       4	&     0.0	&    92.4	&    0.29	& 1052.50	&     0.0\\ 
	A10-3-2-3-2         	& $<0.01$	&  135.00	&       2	&    77.9	&     7.4	& $<0.01$	&  612.00	&       3	&     0.0	&    55.3	& $<0.01$	&  612.00	&       3	&     0.0	&    55.3	&    0.12	&  612.00	&     0.0\\ 
	B1-5-4-3-3          	& $<0.01$	&    0.00	&       2	&   100.0	&     0.0	& $<0.01$	&   97.50	&       5	&    70.9	&    94.8	& $<0.01$	&   97.50	&       5	&    70.9	&    94.8	&\texttt{TL}	&  335.00	&     0.0\\ 
	B2-5-4-3-3          	& $<0.01$	&  300.00	&       2	&    66.9	&    27.6	& $<0.01$	&  905.25	&       5	&     0.0	&    88.1	& $<0.01$	&  905.25	&       5	&     0.0	&    88.1	&\texttt{TL}	&  905.25	&     0.0\\ 
	B3-5-4-3-3          	& $<0.01$	&    0.00	&       2	&   100.0	&     0.0	& $<0.01$	&  159.71	&       4	&     0.0	&    65.6	& $<0.01$	&  159.71	&       8	&     0.0	&    65.6	&\texttt{TL}	&  159.71	&     0.0\\ 
	B4-5-4-3-3          	& $<0.01$	&  668.67	&      18	&     5.5	&    80.5	& $<0.01$	&  674.41	&      12	&     4.7	&    85.6	& $<0.01$	&  675.33	&       6	&     4.5	&    85.5	&    0.72	&  707.33	&     0.0\\ 
	B5-5-4-3-3          	& $<0.01$	&  459.43	&       2	&    16.3	&    33.7	& $<0.01$	&  470.40	&       4	&    14.3	&    76.1	& $<0.01$	&  470.40	&       4	&    14.3	&    76.1	&\texttt{TL}	&  548.57	&     0.0\\ 
	B6-5-4-3-3          	& $<0.01$	&  795.00	&       2	&    23.7	&    63.4	& $<0.01$	&  823.20	&       6	&    21.0	&    85.9	& $<0.01$	&  823.20	&       8	&    21.0	&    90.2	&\texttt{TL}	& 1042.55	&     0.0\\ 
	B7-5-4-3-3          	& $<0.01$	&  638.00	&       3	&    57.7	&    52.9	& $<0.01$	& 1028.00	&       4	&    31.9	&    89.4	& $<0.01$	& 1028.00	&       5	&    31.9	&    93.5	&  336.62	& 1509.00	&     0.0\\ 
	B8-5-4-3-3          	& $<0.01$	&  663.00	&       2	&    33.7	&    49.1	& $<0.01$	& 1000.00	&       8	&     0.0	&    93.9	& $<0.01$	& 1000.00	&      12	&     0.0	&    93.9	&\texttt{TL}	& 1000.00	&     0.0\\ 
	B9-5-4-3-3          	& $<0.01$	&  352.00	&       2	&    29.6	&    24.6	& $<0.01$	&  499.69	&       7	&     0.0	&    95.7	& $<0.01$	&  499.69	&       7	&     0.0	&    95.7	&\texttt{TL}	&  499.69	&     0.0\\ 
	B10-5-4-3-3         	& $<0.01$	&  827.40	&       8	&     4.7	&    67.8	& $<0.01$	&  868.00	&      14	&     0.0	&    85.9	& $<0.01$	&  868.00	&      10	&     0.0	&    82.6	&    0.01	&  868.00	&     0.0\\ 
	C1-8-6-6-4          	&    0.09	&     ---	&     100	&     ---	&    77.8	&    0.15	& 1828.07	&      83	&     0.0	&    85.2	&    0.02	& 1828.07	&      12	&     0.0	&    85.2	&\texttt{TL}	& 1828.07	&     0.0\\ 
	C2-8-6-6-4          	&    0.17	&     ---	&     100	&     ---	&    47.5	&    0.10	& 1011.54	&      46	&    19.1	&    62.2	&    0.02	& 1011.54	&      14	&    19.1	&    65.4	&\texttt{TL}	& 1250.97	&     0.0\\ 
	C3-8-6-6-4          	&    0.10	&     ---	&     100	&     ---	&    73.2	&    0.02	& 1273.58	&      15	&    23.1	&    65.8	&    0.05	& 1273.58	&      39	&    23.1	&    76.7	&\texttt{TL}	& 1656.52	&     0.0\\ 
	C4-8-6-6-4          	& $<0.01$	&    0.00	&       2	&   100.0	&     0.0	& $<0.01$	&  220.80	&       6	&     6.2	&   100.0	& $<0.01$	&  220.80	&       5	&     6.2	&   100.0	&\texttt{TL}	&  235.42	&     0.0\\ 
	C5-8-6-6-4          	&    0.12	&     ---	&     100	&     ---	&    22.9	&    0.02	&  653.55	&      16	&     0.0	&    43.0	&    0.04	&  653.55	&      33	&     0.0	&    44.9	&\texttt{TL}	&  652.48	&     0.2\\ 
	C6-8-6-6-4          	& $<0.01$	& 1056.49	&       2	&     2.7	&    40.8	& $<0.01$	& 1085.33	&       6	&     0.0	&    91.3	& $<0.01$	& 1085.33	&       6	&     0.0	&    91.3	&\texttt{TL}	& 1082.92	&     0.2\\ 
	C7-8-6-6-4          	&    0.03	&  584.29	&      26	&    48.7	&    30.1	&    0.03	& 1139.93	&      28	&     0.0	&    62.8	&    0.01	& 1139.93	&      10	&     0.0	&    60.9	&\texttt{TL}	& 1132.68	&     0.6\\ 
	C8-8-6-6-4          	&    0.02	& 1303.20	&      19	&    21.3	&    63.3	&    0.03	& 1303.20	&      31	&    21.3	&    69.7	&    0.02	& 1303.20	&      17	&    21.3	&    71.7	&\texttt{TL}	& 1655.83	&     0.0\\ 
	C9-8-6-6-4          	&    0.05	&  780.13	&      36	&    39.8	&    64.7	& $<0.01$	& 1295.22	&       5	&     0.0	&    80.7	& $<0.01$	& 1192.98	&       4	&     7.9	&    79.4	&\texttt{TL}	& 1292.55	&     0.2\\ 
	C10-8-6-6-4         	& $<0.01$	& 2635.85	&       5	&     2.3	&    78.1	&    0.04	& 2698.76	&      35	&     0.0	&    85.8	&    0.02	& 2698.76	&      13	&     0.0	&    85.6	&\texttt{TL}	& 2698.76	&     0.0\\ 
	D1-12-10-8-5        	&    0.18	&     ---	&     100	&     ---	&    99.8	&    0.01	& 3006.86	&       6	&     0.2	&    99.9	&    0.01	& 3006.86	&       6	&     0.2	&    99.9	&   13.77	& 3013.00	&     0.0\\ 
	D2-12-10-8-5        	&    0.04	&  540.00	&      14	&    60.4	&    30.5	&    0.01	&  540.00	&       7	&    60.4	&    54.7	&    0.04	&  540.00	&       9	&    60.4	&    69.2	&\texttt{TL}	& 1364.01	&     0.0\\ 
	D3-12-10-8-5        	&    0.25	&     ---	&     100	&     ---	&    83.0	&    0.16	& 2343.06	&      50	&     7.1	&    83.4	&    0.35	& 2521.03	&     100	&     0.0	&    84.9	&\texttt{TL}	& 2521.03	&     0.0\\ 
	D4-12-10-8-5        	&    0.31	&     ---	&     100	&     ---	&    10.3	& $<0.01$	&  270.05	&       5	&    73.4	&    44.9	&    0.02	&  270.05	&      11	&    73.4	&    45.5	&\texttt{TL}	& 1013.98	&     0.0\\ 
	D5-12-10-8-5        	&    0.26	&     ---	&     100	&     ---	&    77.1	&    0.47	&     ---	&     100	&     ---	&    87.6	&    0.51	&     ---	&     100	&     ---	&    87.6	&\texttt{TL}	& 2949.09	&     0.0\\ 
	D6-12-10-8-5        	&    0.13	& 1114.17	&      33	&    41.4	&    41.4	&    0.49	&     ---	&     100	&     ---	&    72.4	&    0.63	&     ---	&     100	&     ---	&    72.5	&\texttt{TL}	& 1901.61	&     0.0\\ 
	D7-12-10-8-5        	&    0.27	&     ---	&     100	&     ---	&    59.0	&    0.40	&     ---	&     100	&     ---	&    63.1	&    0.48	&     ---	&     100	&     ---	&    62.5	&\texttt{TL}	& 2204.61	&     0.0\\ 
	D8-12-10-8-5        	&    0.16	&     ---	&     100	&     ---	&    69.5	&    0.24	&     ---	&     100	&     ---	&    91.6	&    0.03	& 3500.48	&      10	&     0.0	&    94.2	&\texttt{TL}	& 3357.58	&     4.1\\ 
	D9-12-10-8-5        	&    0.18	&     ---	&     100	&     ---	&    20.1	&    0.03	&  785.53	&      16	&     0.0	&    34.8	&    0.08	&  785.53	&      44	&     0.0	&    34.9	&\texttt{TL}	&  782.89	&     0.3\\ 
	D10-12-10-8-5       	&    0.22	&     ---	&     100	&     ---	&    71.8	&    0.07	& 3025.13	&      41	&     0.0	&    79.8	&    0.13	& 3025.13	&      49	&     0.0	&    86.9	&\texttt{TL}	& 3025.13	&     0.0\\ 
	E1-10-10-15-12      	&    0.10	&  317.47	&      17	&     3.3	&     4.4	&    0.02	&  328.17	&       5	&     0.0	&    43.1	&    0.04	&  328.17	&       8	&     0.0	&    44.9	&\texttt{TL}	&  328.17	&     0.0\\ 
	E2-10-10-15-12      	& $<0.01$	&    0.00	&       2	&   100.0	&     0.0	&    0.03	&  728.99	&       5	&     0.0	&   100.0	&    0.05	&  728.99	&       7	&     0.0	&   100.0	&\texttt{TL}	&  331.50	&    54.5\\ 
	E3-10-10-15-12      	& $<0.01$	&    0.00	&       2	&   100.0	&     0.0	&    0.03	&    0.00	&       5	&   100.0	&    88.8	&    0.08	&  171.37	&       6	&     0.0	&    88.8	&\texttt{TL}	&  171.37	&     0.0\\ 
	E4-10-10-15-12      	& $<0.01$	&  260.67	&       2	&     1.4	&     4.4	&    0.05	&  264.50	&       9	&     0.0	&    83.7	&    0.08	&  264.50	&       7	&     0.0	&    83.7	&\texttt{TL}	&  264.50	&     0.0\\ 
	E5-10-10-15-12      	&    0.75	&     ---	&     100	&     ---	&     6.0	&    0.03	&  396.41	&       5	&     0.0	&    38.1	&    0.07	&  396.41	&      11	&     0.0	&    41.9	&\texttt{TL}	&  383.20	&     3.3\\ 
	E6-10-10-15-12      	& $<0.01$	&    0.00	&       2	&   100.0	&     0.0	&    0.04	&    0.00	&       5	&   100.0	&   100.0	&    0.06	&  498.72	&       7	&     0.0	&   100.0	&\texttt{TL}	&  498.72	&     0.0\\ 
	E7-10-10-15-12      	& $<0.01$	&    0.00	&       2	&   100.0	&     0.0	&    0.03	&    0.00	&       6	&   100.0	&   100.0	&    0.06	&  635.51	&       9	&     0.0	&   100.0	&\texttt{TL}	&  635.51	&     0.0\\ 
	E8-10-10-15-12      	&    0.05	&  386.22	&       7	&    56.0	&     9.9	&    0.04	&  386.22	&       5	&    56.0	&    61.2	&    0.14	&  876.89	&      19	&     0.0	&    67.3	&\texttt{TL}	&  876.89	&     0.0\\ 
	E9-10-10-15-12      	&    0.01	&    0.00	&       2	&   100.0	&     0.0	&    0.03	&  323.88	&       4	&     0.0	&    89.7	&    0.07	&  323.88	&      12	&     0.0	&    89.7	&\texttt{TL}	&    0.00	&   100.0\\ 
	E10-10-10-15-12     	& $<0.01$	&  299.22	&       2	&     3.6	&     5.0	&    0.02	&  310.37	&       5	&     0.0	&    92.9	&    0.04	&  310.37	&       7	&     0.0	&    92.9	&\texttt{TL}	&  310.37	&     0.0\\ 
	\hline
	\multicolumn{1}{l}{\tblAve}
	&    0.07	&        	&    30.8	&    35.4	&    37.7	&    0.05	&        	&    19.4	&    19.1	&    75.1	&    0.07	&        	&    19.7	&     7.3	&    79.3	&  440.33	&        	&     3.3\\ 
	\bottomrule
\end{tabular}

	}
\end{table}

\begin{figure}[htbp]
	\centering
	\includegraphics[width=.7\linewidth]{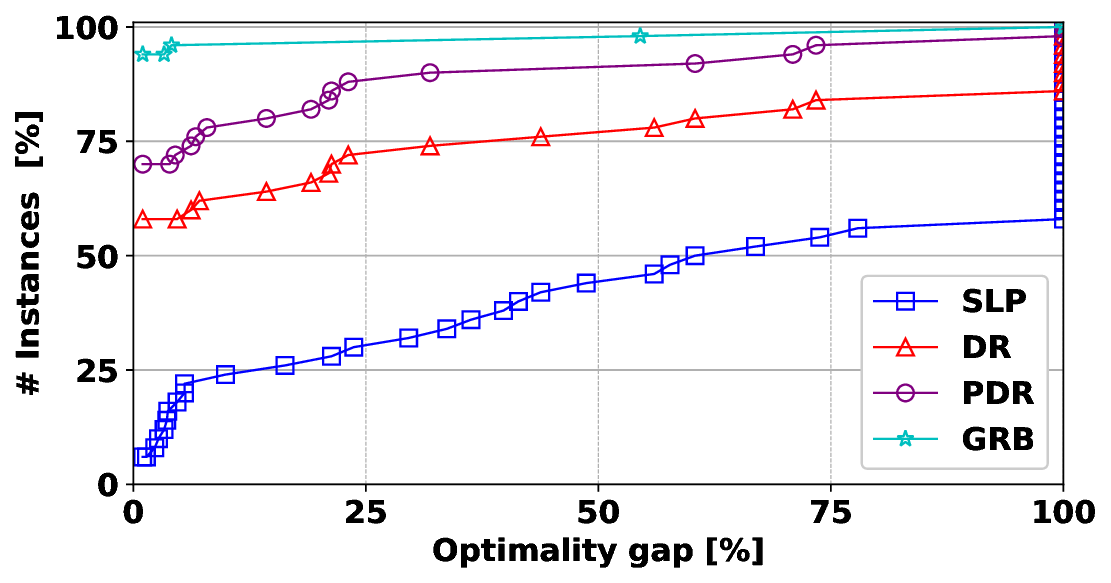}
	\caption{Performance profiles of the optimality gap of the feasible solution returned by settings \testSLP, \testDR, \testPDR, and \testGRB.}
	\label{fig:gap}
\end{figure}

From \cref{table:SLP-YSLP-PYSLP-random}, we can further confirm that \testPDR outperforms both \testDR and \testSLP in terms of finding high-quality solutions on these randomly generated instances.
In particular, compared to those returned \testSLP and \testDR, the average objective ratios $\frac{o_t}{o_0} $ returned by \testPDR are generally much larger, thereby rendering \testPDR to find better feasible solutions.
This observation is more intuitively depicted in \cref{fig:gap} where the purple-circle line corresponding to setting \testPDR is much higher than red-{triangle} and blue-square lines corresponding to settings \testDR and \testSLP, respectively.   

Another observation from \cref{table:SLP-YSLP-PYSLP-random} and \cref{fig:gap} is that 

\testPDR is able to efficiently find high-quality feasible solutions.
Indeed, 
(i) the average CPU time returned by \testPDR is only $0.07$ seconds, which is much smaller than that returned by \testGRB;
and 
(ii) \testPDR can find the optimal solution for a fairly large number of instances (more than $70\%$ instances, as shown in \cref{fig:gap}), and for instances \texttt{D8}, \texttt{E2}, \texttt{E5}, and \texttt{E9}, \testDR can even find much better solutions than \testGRB.

These results shed an useful insight that the proposed \PDR algorithm can potentially be embedded into a global solver as a heuristic component to enhance the solver's capabilities (as they can provide better lower bounds to help prune the branch-and-bound tree).

\section{Conclusions}\label{sect:conclusion}

In this paper, we have proposed a new \NLP formulation for the pooling problem, which involves only the flow variables and thus is much more amenable to the algorithmic design. 
In particular, we have shown that the well-known \DR algorithm can be seen as a direct application of the \SLP algorithm to the newly proposed formulation.
With this new useful insight, we have developed a new variant of \DR algorithm, called \PDR algorithm, based on the proposed \NLP formulation. 
The proposed \PDR algorithm is a penalty algorithm where the violations of the (linearized) nonlinear constraints are penalized in the objective function of the \LP approximation problem with the penalty terms increasing when the constraint violations tend to be large.

Moreover, the \LP approximation problems in the proposed \PDR algorithm can construct a sequence of iterates (i) for which all linear constraints in the \NLP formulation are satisfied and (ii) whose objective values are generally larger than the objective values of those constructed in the classic \DR algorithm.
This feature makes the \PDR algorithm more likely to find a better feasible solution for the pooling problem.
By computational experiments, we show the advantage of the proposed \PDR over the classic \SLP and \DR algorithms in terms of finding a high-quality solution for the pooling problem.

\bibliography{oil}

\begin{thebibliography}{45}
\expandafter\ifx\csname natexlab\endcsname\relax\def\natexlab#1{#1}\fi
\providecommand{\url}[1]{\texttt{#1}}
\providecommand{\href}[2]{#2}
\providecommand{\path}[1]{#1}
\providecommand{\DOIprefix}{doi:}
\providecommand{\ArXivprefix}{arXiv:}
\providecommand{\URLprefix}{URL: }
\providecommand{\Pubmedprefix}{pmid:}
\providecommand{\doi}[1]{\href{http://dx.doi.org/#1}{\path{#1}}}
\providecommand{\Pubmed}[1]{\href{pmid:#1}{\path{#1}}}
\providecommand{\bibinfo}[2]{#2}
\ifx\xfnm\relax \def\xfnm[#1]{\unskip,\space#1}\fi
\bibitem[{Adhya et~al.(1999)Adhya, Tawarmalani \& Sahinidis}]{Adhya1999}
\bibinfo{author}{Adhya, N.}, \bibinfo{author}{Tawarmalani, M.}, \&
  \bibinfo{author}{Sahinidis, N.~V.} (\bibinfo{year}{1999}).
\newblock \bibinfo{title}{A {Lagrangian} approach to the pooling problem}.
\newblock {\it \bibinfo{journal}{Industrial \& Engineering Chemistry
  Research}\/},  {\it \bibinfo{volume}{38}\/}, \bibinfo{pages}{1956--1972}.
\bibitem[{Alfaki \& Haugland(2013{\natexlab{a}})}]{Alfaki2013a}
\bibinfo{author}{Alfaki, M.}, \& \bibinfo{author}{Haugland, D.}
  (\bibinfo{year}{2013}{\natexlab{a}}).
\newblock \bibinfo{title}{A multi-commodity flow formulation for the
  generalized pooling problem}.
\newblock {\it \bibinfo{journal}{Journal of Global Optimization}\/},  {\it
  \bibinfo{volume}{56}\/}, \bibinfo{pages}{917--937}.
\bibitem[{Alfaki \& Haugland(2013{\natexlab{b}})}]{Alfaki2013}
\bibinfo{author}{Alfaki, M.}, \& \bibinfo{author}{Haugland, D.}
  (\bibinfo{year}{2013}{\natexlab{b}}).
\newblock \bibinfo{title}{Strong formulations for the pooling problem}.
\newblock {\it \bibinfo{journal}{Journal of Global Optimization}\/},  {\it
  \bibinfo{volume}{56}\/}, \bibinfo{pages}{897--916}.
\bibitem[{Alfaki \& Haugland(2014)}]{Alfaki2014}
\bibinfo{author}{Alfaki, M.}, \& \bibinfo{author}{Haugland, D.}
  (\bibinfo{year}{2014}).
\newblock \bibinfo{title}{A cost minimization heuristic for the pooling
  problem}.
\newblock {\it \bibinfo{journal}{Annals of Operations Research}\/},  {\it
  \bibinfo{volume}{222}\/}, \bibinfo{pages}{73--87}.
\bibitem[{Almutairi \& Elhedhli(2009)}]{Almutairi2009}
\bibinfo{author}{Almutairi, H.}, \& \bibinfo{author}{Elhedhli, S.}
  (\bibinfo{year}{2009}).
\newblock \bibinfo{title}{A new lagrangean approach to the pooling problem}.
\newblock {\it \bibinfo{journal}{Journal of Global Optimization}\/},  {\it
  \bibinfo{volume}{45}\/}, \bibinfo{pages}{237--257}.
\bibitem[{Amos et~al.(1997)Amos, R{\"o}nnqvist \& Gill}]{Amos1997}
\bibinfo{author}{Amos, F.}, \bibinfo{author}{R{\"o}nnqvist, M.}, \&
  \bibinfo{author}{Gill, G.} (\bibinfo{year}{1997}).
\newblock \bibinfo{title}{Modelling the pooling problem at the {New Zealand
  Refining Company}}.
\newblock {\it \bibinfo{journal}{Journal of the Operational Research
  Society}\/},  {\it \bibinfo{volume}{48}\/}, \bibinfo{pages}{767--778}.
\bibitem[{Aspen~Technology(2022)}]{PIMS}
\bibinfo{author}{Aspen~Technology, I.} (\bibinfo{year}{2022}).
\newblock \bibinfo{title}{{Aspen PIMS: Advanced Optimization Features}}.
\newblock
  \bibinfo{note}{\url{https://esupport.aspentech.com/UniversityCourse?Id=a3p0B0000015XK2QAM}}.
\bibitem[{Audet et~al.(2004)Audet, Brimberg, Hansen, Digabel \&
  Mladenovi{\'c}}]{Audet2004}
\bibinfo{author}{Audet, C.}, \bibinfo{author}{Brimberg, J.},
  \bibinfo{author}{Hansen, P.}, \bibinfo{author}{Digabel, S.~L.}, \&
  \bibinfo{author}{Mladenovi{\'c}, N.} (\bibinfo{year}{2004}).
\newblock \bibinfo{title}{Pooling problem: Alternate formulations and solution
  methods}.
\newblock {\it \bibinfo{journal}{Management Science}\/},  {\it
  \bibinfo{volume}{50}\/}, \bibinfo{pages}{761--776}.
\bibitem[{Baker \& Lasdon(1985)}]{Baker1985}
\bibinfo{author}{Baker, T.~E.}, \& \bibinfo{author}{Lasdon, L.~S.}
  (\bibinfo{year}{1985}).
\newblock \bibinfo{title}{Successive linear programming at {Exxon}}.
\newblock {\it \bibinfo{journal}{Management Science}\/},  {\it
  \bibinfo{volume}{31}\/}, \bibinfo{pages}{264--274}.
\bibitem[{{Ben-Tal} et~al.(1994){Ben-Tal}, Eiger \& Gershovitz}]{BenTal1994}
\bibinfo{author}{{Ben-Tal}, A.}, \bibinfo{author}{Eiger, G.}, \&
  \bibinfo{author}{Gershovitz, V.} (\bibinfo{year}{1994}).
\newblock \bibinfo{title}{Global minimization by reducing the duality gap}.
\newblock {\it \bibinfo{journal}{Mathematical Programming}\/},  {\it
  \bibinfo{volume}{63}\/}, \bibinfo{pages}{193--212}.
\bibitem[{Blom et~al.(2014)Blom, Burt, Pearce \& Stuckey}]{Blom2014}
\bibinfo{author}{Blom, M.~L.}, \bibinfo{author}{Burt, C.~N.},
  \bibinfo{author}{Pearce, A.~R.}, \& \bibinfo{author}{Stuckey, P.~J.}
  (\bibinfo{year}{2014}).
\newblock \bibinfo{title}{A decomposition-based heuristic for collaborative
  scheduling in a network of open-pit mines}.
\newblock {\it \bibinfo{journal}{INFORMS Journal on Computing}\/},  {\it
  \bibinfo{volume}{26}\/}, \bibinfo{pages}{658--676}.
\bibitem[{Boland et~al.(2017)Boland, Kalinowski \& Rigterink}]{Boland2017}
\bibinfo{author}{Boland, N.}, \bibinfo{author}{Kalinowski, T.}, \&
  \bibinfo{author}{Rigterink, F.} (\bibinfo{year}{2017}).
\newblock \bibinfo{title}{A polynomially solvable case of the pooling problem}.
\newblock {\it \bibinfo{journal}{Journal of Global Optimization}\/},  {\it
  \bibinfo{volume}{67}\/}, \bibinfo{pages}{621--630}.
\bibitem[{Castro(2023)}]{Castro2023}
\bibinfo{author}{Castro, P.~M.} (\bibinfo{year}{2023}).
\newblock \bibinfo{title}{Global optimization of {QCP}s using {MIP} relaxations
  with a base-2 logarithmic partitioning scheme}.
\newblock {\it \bibinfo{journal}{Industrial \& Engineering Chemistry
  Research}\/},  {\it \bibinfo{volume}{62}\/}, \bibinfo{pages}{11053--11066}.
\bibitem[{Dai et~al.(2018)Dai, Diao \& Fu}]{Dai2018}
\bibinfo{author}{Dai, Y.-H.}, \bibinfo{author}{Diao, R.}, \&
  \bibinfo{author}{Fu, K.} (\bibinfo{year}{2018}).
\newblock \bibinfo{title}{Complexity analysis and algorithm design of pooling
  problem}.
\newblock {\it \bibinfo{journal}{Journal of the Operations Research Society of
  China}\/},  {\it \bibinfo{volume}{6}\/}, \bibinfo{pages}{249--266}.
\bibitem[{DeWitt et~al.(1989)DeWitt, Lasdon, Waren, Brenner \&
  Melhem}]{DeWitt1989}
\bibinfo{author}{DeWitt, C.~W.}, \bibinfo{author}{Lasdon, L.~S.},
  \bibinfo{author}{Waren, A.~D.}, \bibinfo{author}{Brenner, D.~A.}, \&
  \bibinfo{author}{Melhem, S.~A.} (\bibinfo{year}{1989}).
\newblock \bibinfo{title}{Omega: An improved gasoline blending system for
  texaco}.
\newblock {\it \bibinfo{journal}{Interfaces}\/},  {\it \bibinfo{volume}{19}\/},
  \bibinfo{pages}{85--101}.
\bibitem[{Dey \& Gupte(2015)}]{Dey2015}
\bibinfo{author}{Dey, S.~S.}, \& \bibinfo{author}{Gupte, A.}
  (\bibinfo{year}{2015}).
\newblock \bibinfo{title}{Analysis of {MILP} techniques for the pooling
  problem}.
\newblock {\it \bibinfo{journal}{Operations Research}\/},  {\it
  \bibinfo{volume}{63}\/}, \bibinfo{pages}{412--427}.
\bibitem[{Fieldhouse(1993)}]{Fieldhouse1993}
\bibinfo{author}{Fieldhouse, M.} (\bibinfo{year}{1993}).
\newblock \bibinfo{title}{The pooling problem}.
\newblock In \bibinfo{editor}{T.~Ciriani}, \& \bibinfo{editor}{R.~Leachman}
  (Eds.), {\it \bibinfo{booktitle}{Optimization in Industry: Mathematical
  Programming and Modeling Techniques in Practice}\/} (pp.
  \bibinfo{pages}{223--230}).
\bibitem[{Floudas \& Visweswaran(1990)}]{Floudas1990}
\bibinfo{author}{Floudas, C.}, \& \bibinfo{author}{Visweswaran, V.}
  (\bibinfo{year}{1990}).
\newblock \bibinfo{title}{A global optimization algorithm ({GOP}) for certain
  classes of nonconvex {NLP}s---{I}. {T}heory}.
\newblock {\it \bibinfo{journal}{Computers \& Chemical Engineering}\/},  {\it
  \bibinfo{volume}{14}\/}, \bibinfo{pages}{1397--1417}.
\bibitem[{Floudas \& Aggarwal(1990)}]{Floudas1990a}
\bibinfo{author}{Floudas, C.~A.}, \& \bibinfo{author}{Aggarwal, A.}
  (\bibinfo{year}{1990}).
\newblock \bibinfo{title}{A decomposition strategy for global optimum search in
  the pooling problem}.
\newblock {\it \bibinfo{journal}{ORSA Journal on Computing}\/},  {\it
  \bibinfo{volume}{2}\/}, \bibinfo{pages}{225--235}.
\bibitem[{Foulds et~al.(1992)Foulds, Haugland \& J{\"O}rnsten}]{Foulds1992}
\bibinfo{author}{Foulds, L.}, \bibinfo{author}{Haugland, D.}, \&
  \bibinfo{author}{J{\"O}rnsten, K.} (\bibinfo{year}{1992}).
\newblock \bibinfo{title}{A bilinear approach to the pooling problem}.
\newblock {\it \bibinfo{journal}{Optimization}\/},  {\it
  \bibinfo{volume}{24}\/}, \bibinfo{pages}{165--180}.
\bibitem[{Galan \& Grossmann(1998)}]{Galan1998}
\bibinfo{author}{Galan, B.}, \& \bibinfo{author}{Grossmann, I.~E.}
  (\bibinfo{year}{1998}).
\newblock \bibinfo{title}{Optimal design of distributed wastewater treatment
  networks}.
\newblock {\it \bibinfo{journal}{Industrial \& Engineering Chemistry
  Research}\/},  {\it \bibinfo{volume}{37}\/}, \bibinfo{pages}{4036--4048}.
\bibitem[{Greenberg(1995)}]{Greenberg1995}
\bibinfo{author}{Greenberg, H.~J.} (\bibinfo{year}{1995}).
\newblock \bibinfo{title}{Analyzing the pooling problem}.
\newblock {\it \bibinfo{journal}{ORSA Journal on Computing}\/},  {\it
  \bibinfo{volume}{7}\/}, \bibinfo{pages}{205--217}.
\bibitem[{Griffith \& Stewart(1961)}]{Griffith1961}
\bibinfo{author}{Griffith, R.~E.}, \& \bibinfo{author}{Stewart, R.~A.}
  (\bibinfo{year}{1961}).
\newblock \bibinfo{title}{A nonlinear programming technique for the
  optimization of continuous processing systems}.
\newblock {\it \bibinfo{journal}{Management Science}\/},  {\it
  \bibinfo{volume}{7}\/}, \bibinfo{pages}{379--392}.
\bibitem[{Grothey \& McKinnon(2023)}]{Grothey2023}
\bibinfo{author}{Grothey, A.}, \& \bibinfo{author}{McKinnon, K.}
  (\bibinfo{year}{2023}).
\newblock \bibinfo{title}{On the effectiveness of sequential linear programming
  for the pooling problem}.
\newblock {\it \bibinfo{journal}{Annals of Operations Research}\/},  {\it
  \bibinfo{volume}{322}\/}, \bibinfo{pages}{691--711}.
\bibitem[{Gupte et~al.(2017)Gupte, Ahmed, Dey \& Cheon}]{Gupte2017}
\bibinfo{author}{Gupte, A.}, \bibinfo{author}{Ahmed, S.}, \bibinfo{author}{Dey,
  S.~S.}, \& \bibinfo{author}{Cheon, M.~S.} (\bibinfo{year}{2017}).
\newblock \bibinfo{title}{Relaxations and discretizations for the pooling
  problem}.
\newblock {\it \bibinfo{journal}{Journal of Global Optimization}\/},  {\it
  \bibinfo{volume}{67}\/}, \bibinfo{pages}{631--669}.
\bibitem[{Haugland(2010)}]{Haugland2010}
\bibinfo{author}{Haugland, D.} (\bibinfo{year}{2010}).
\newblock \bibinfo{title}{An overview of models and solution methods for
  pooling problems}.
\newblock In \bibinfo{editor}{E.~Bj{\o}rndal},
  \bibinfo{editor}{M.~Bj{\o}rndal}, \bibinfo{editor}{P.~M. Pardalos}, \&
  \bibinfo{editor}{M.~R{\"o}nnqvist} (Eds.), {\it \bibinfo{booktitle}{Energy,
  Natural Resources and Environmental Economics}\/} (pp.
  \bibinfo{pages}{459--469}).
\bibitem[{Haverly(1978)}]{Haverly1978}
\bibinfo{author}{Haverly, C.~A.} (\bibinfo{year}{1978}).
\newblock \bibinfo{title}{Studies of the behavior of recursion for the pooling
  problem}.
\newblock {\it \bibinfo{journal}{ACM SIGMAP Bulletin}\/},  (pp.
  \bibinfo{pages}{19--28}).
\bibitem[{{Haverly Systems, Inc.}(2022)}]{Kathy2022}
\bibinfo{author}{{Haverly Systems, Inc.}} (\bibinfo{year}{2022}).
\newblock \bibinfo{title}{Deriving distributed recursion}.
\newblock
  \bibinfo{note}{\url{https://www.haverly.com/kathy-blog/759-blog-82-derivedr}}.
\bibitem[{Kammammettu \& Li(2020)}]{Kammammettu2020}
\bibinfo{author}{Kammammettu, S.}, \& \bibinfo{author}{Li, Z.}
  (\bibinfo{year}{2020}).
\newblock \bibinfo{title}{Two-stage robust optimization of water treatment
  network design and operations under uncertainty}.
\newblock {\it \bibinfo{journal}{Industrial \& Engineering Chemistry
  Research}\/},  {\it \bibinfo{volume}{59}\/}, \bibinfo{pages}{1218--1233}.
\bibitem[{Khor \& Varvarezos(2017)}]{Khor2017}
\bibinfo{author}{Khor, C.~S.}, \& \bibinfo{author}{Varvarezos, D.}
  (\bibinfo{year}{2017}).
\newblock \bibinfo{title}{Petroleum refinery optimization}.
\newblock {\it \bibinfo{journal}{Optimization and Engineering}\/},  {\it
  \bibinfo{volume}{18}\/}, \bibinfo{pages}{943--989}.
\bibitem[{Kutz et~al.(2014)Kutz, Davis, Creek, Kenaston, Stenstrom \&
  Connor}]{Kutz2014}
\bibinfo{author}{Kutz, T.}, \bibinfo{author}{Davis, M.},
  \bibinfo{author}{Creek, R.}, \bibinfo{author}{Kenaston, N.},
  \bibinfo{author}{Stenstrom, C.}, \& \bibinfo{author}{Connor, M.}
  (\bibinfo{year}{2014}).
\newblock \bibinfo{title}{Optimizing {Chevron}'s refineries}.
\newblock {\it \bibinfo{journal}{Interfaces}\/},  {\it \bibinfo{volume}{44}\/},
  \bibinfo{pages}{39--54}.
\bibitem[{Lasdon \& Joffe(1990)}]{Lasdon1990}
\bibinfo{author}{Lasdon, L.~S.}, \& \bibinfo{author}{Joffe, B.}
  (\bibinfo{year}{1990}).
\newblock {\it \bibinfo{title}{The relationship between distributive recursion
  and successive linear programming in refining production planning models}\/}.
\newblock \bibinfo{publisher}{National Petroleum Refiners Association}.
\bibitem[{Lasdon et~al.(1979)Lasdon, Waren, Sarkar \& Palacios}]{Lasdon1979}
\bibinfo{author}{Lasdon, L.~S.}, \bibinfo{author}{Waren, A.~D.},
  \bibinfo{author}{Sarkar, S.}, \& \bibinfo{author}{Palacios, F.}
  (\bibinfo{year}{1979}).
\newblock \bibinfo{title}{Solving the pooling problem using generalized reduced
  gradient and successive linear programming algorithms}.
\newblock {\it \bibinfo{journal}{ACM SIGMAP Bulletin}\/},  {\it
  \bibinfo{volume}{27}\/}, \bibinfo{pages}{9--15}.
\bibitem[{Luenberger \& Ye(2021)}]{Luenberger2021}
\bibinfo{author}{Luenberger, D.~G.}, \& \bibinfo{author}{Ye, Y.}
  (\bibinfo{year}{2021}).
\newblock {\it \bibinfo{title}{Linear and Nonlinear Programming}\/}.
\newblock \bibinfo{address}{Cham}: \bibinfo{publisher}{Springer International
  Publishing}.
\bibitem[{Meyer \& Floudas(2006)}]{Meyer2006}
\bibinfo{author}{Meyer, C.~A.}, \& \bibinfo{author}{Floudas, C.~A.}
  (\bibinfo{year}{2006}).
\newblock \bibinfo{title}{Global optimization of a combinatorially complex
  generalized pooling problem}.
\newblock {\it \bibinfo{journal}{AIChE Journal}\/},  {\it
  \bibinfo{volume}{52}\/}, \bibinfo{pages}{1027--1037}.
\bibitem[{Misener \& Floudas(2010)}]{Misener2010}
\bibinfo{author}{Misener, R.}, \& \bibinfo{author}{Floudas, C.~A.}
  (\bibinfo{year}{2010}).
\newblock \bibinfo{title}{Global optimization of large-scale generalized
  pooling problems: Quadratically constrained {MINLP} models}.
\newblock {\it \bibinfo{journal}{Industrial \& Engineering Chemistry
  Research}\/},  {\it \bibinfo{volume}{49}\/}, \bibinfo{pages}{5424--5438}.
\bibitem[{Misener et~al.(2011)Misener, Thompson \& Floudas}]{Misener2011}
\bibinfo{author}{Misener, R.}, \bibinfo{author}{Thompson, J.~P.}, \&
  \bibinfo{author}{Floudas, C.~A.} (\bibinfo{year}{2011}).
\newblock \bibinfo{title}{{APOGEE}: Global optimization of standard,
  generalized, and extended pooling problems via linear and logarithmic
  partitioning schemes}.
\newblock {\it \bibinfo{journal}{Computers \& Chemical Engineering}\/},  {\it
  \bibinfo{volume}{35}\/}, \bibinfo{pages}{876--892}.
\bibitem[{Nocedal \& Wright(1999)}]{Nocedal1999}
\bibinfo{author}{Nocedal, J.}, \& \bibinfo{author}{Wright, S.~J.}
  (\bibinfo{year}{1999}).
\newblock {\it \bibinfo{title}{Numerical optimization}\/}.
\newblock \bibinfo{publisher}{Springer}.
\bibitem[{Pham et~al.(2009)Pham, Laird \& {El-Halwagi}}]{Pham2009}
\bibinfo{author}{Pham, V.}, \bibinfo{author}{Laird, C.}, \&
  \bibinfo{author}{{El-Halwagi}, M.} (\bibinfo{year}{2009}).
\newblock \bibinfo{title}{Convex hull discretization approach to the global
  optimization of pooling problems}.
\newblock {\it \bibinfo{journal}{Industrial \& Engineering Chemistry
  Research}\/},  {\it \bibinfo{volume}{48}\/}, \bibinfo{pages}{1973--1979}.
\bibitem[{{R{\'i}os-Mercado} \& {Borraz-S{\'a}nchez}(2015)}]{Rios-Mercado2015}
\bibinfo{author}{{R{\'i}os-Mercado}, R.~Z.}, \&
  \bibinfo{author}{{Borraz-S{\'a}nchez}, C.} (\bibinfo{year}{2015}).
\newblock \bibinfo{title}{Optimization problems in natural gas transportation
  systems: A state-of-the-art review}.
\newblock {\it \bibinfo{journal}{Applied Energy}\/},  {\it
  \bibinfo{volume}{147}\/}, \bibinfo{pages}{536--555}.
\bibitem[{R{\o}mo et~al.(2009)R{\o}mo, Tomasgard, Hellemo, Fodstad, Eidesen \&
  Pedersen}]{Romo2009}
\bibinfo{author}{R{\o}mo, F.}, \bibinfo{author}{Tomasgard, A.},
  \bibinfo{author}{Hellemo, L.}, \bibinfo{author}{Fodstad, M.},
  \bibinfo{author}{Eidesen, B.~H.}, \& \bibinfo{author}{Pedersen, B.}
  (\bibinfo{year}{2009}).
\newblock \bibinfo{title}{Optimizing the {N}orwegian natural gas production and
  transport}.
\newblock {\it \bibinfo{journal}{Interfaces}\/},  {\it \bibinfo{volume}{39}\/},
  \bibinfo{pages}{46--56}.
\bibitem[{Sahinidis \& Tawarmalani(2005)}]{Sahinidis2005}
\bibinfo{author}{Sahinidis, N.~V.}, \& \bibinfo{author}{Tawarmalani, M.}
  (\bibinfo{year}{2005}).
\newblock \bibinfo{title}{Accelerating branch-and-bound through a modeling
  language construct for relaxation-specific constraints}.
\newblock {\it \bibinfo{journal}{Journal of Global Optimization}\/},  {\it
  \bibinfo{volume}{32}\/}, \bibinfo{pages}{259--280}.
\bibitem[{Tawarmalani \& Sahinidis(2002)}]{Tawarmalani2002}
\bibinfo{author}{Tawarmalani, M.}, \& \bibinfo{author}{Sahinidis, N.~V.}
  (\bibinfo{year}{2002}).
\newblock {\it \bibinfo{title}{Convexification and Global Optimization in
  Continuous and Mixed-Integer Nonlinear Programming}\/}.
\newblock \bibinfo{address}{Boston, MA}: \bibinfo{publisher}{Springer US}.
\bibitem[{White \& Trierwiler(1980)}]{White1980}
\bibinfo{author}{White, D.~L.}, \& \bibinfo{author}{Trierwiler, L.~D.}
  (\bibinfo{year}{1980}).
\newblock \bibinfo{title}{Distributive recursion at socal}.
\newblock {\it \bibinfo{journal}{ACM SIGMAP Bulletin}\/},  (pp.
  \bibinfo{pages}{22--38}).
\bibitem[{Zhang et~al.(1985)Zhang, Kim \& Lasdon}]{Zhang1985}
\bibinfo{author}{Zhang, J.}, \bibinfo{author}{Kim, N.-H.}, \&
  \bibinfo{author}{Lasdon, L.} (\bibinfo{year}{1985}).
\newblock \bibinfo{title}{An improved successive linear programming algorithm}.
\newblock {\it \bibinfo{journal}{Management Science}\/},  {\it
  \bibinfo{volume}{31}\/}, \bibinfo{pages}{1312--1331}.

\end{thebibliography}
\newpage

\end{document}